\documentclass[12pt]{article}

\usepackage{amsmath,amsthm,amsfonts,amssymb, color}

\usepackage{tikz-cd}  %used for the diagram

\newtheorem{theorem}{Theorem}[section]

\newtheorem{lemma}[theorem]{Lemma}
\newtheorem{definition}[theorem]{Definition}
\newtheorem{remark}[theorem]{Remark}

\def\cF{\mathcal{F}}

\def\cH{\mathcal{H}}

\def\cS{\mathcal{S}}

\def\bC{\mathbb{C}}

\def\bE{\mathbb{E}}
\def\bN{\mathbb{N}}
\def\bP{\mathbb{P}}

\def\bR{\mathbb{R}}

\begin{document}

\title{Parabolic Anderson model with rough initial condition: continuity in law of the solution}

\author{Xiao Liang\footnote{University of Ottawa, Department of Mathematics and Statistics, 150 Louis Pasteur Private, Ottawa, Ontario, K1G 0P8, Canada. E-mail address: tlian081@uottawa.ca}
}

\date{December, 2023}
\maketitle

\begin{abstract}
\noindent The parabolic Anderson model (PAM) is one of the most interesting and challenging SPDEs related to various physical phenomena, and can be described mathematically as a stochastic heat equation driven by linear multiplicative noise. In this paper, we consider PAM with initial condition given by a signed Borel measure on $\mathbb{R}^d$. The forcing term under investigation is examined in two cases: (i) the regular noise, with the spatial covariance given by the Riesz kernel of order $\alpha \in (0,d)$ in spatial dimension $d\geq 1$; (ii) the rough noise, which is a fractional noise in space with Hurst index $H<1/2$ and $d=1$. In both cases, the noise is assumed to be colored in time and we consider a general initial condition. The objective of this article is to show that the solution is continuous in law with respect to the spatial noise parameter. The similar problem for the constant initial condition has been recently studied in \cite{BL2023}. 
\end{abstract}

\noindent {\em MSC 2020:} Primary 60H15; Secondary 60G60, 60H07
%60H15=SPDEs
%60G60=random fields
%60H07=stochastic calculus of variations and the Malliavin calculus

\vspace{1mm}

\noindent {\em Keywords:} stochastic partial differential equations, random fields, rough noise, rough initial condition

%%%%%%%%%%%%%%%%%%%%%%%%%%%%%%%%%%%%%%%%%%%%%%%%
\section{Introduction and Main Results}
In this article, we consider the following stochastic heat equation:
\begin{equation}
\label{PAM}
\begin{cases}
    \dfrac{\partial u}{\partial t}(t,x) = \dfrac{1}{2} \dfrac{\partial^2 u}{\partial x^2}(t,x) + u(t,x)\dot{W}(t,x) , \quad t > 0, x \in \bR^d\\[0.8em]
    u(0,\cdot) = u_0(\cdot), 
  \end{cases}
\end{equation}
with initial condition given by a signed Borel measure $u_0$ on $\bR^d$ such that
\begin{equation}
\label{eqn1-20230522-11:45am}
\int_{\bR^d} e^{-a|x|^2} | u_0 | (\mathrm{d}x) < \infty, \text{ for all } a > 0.
\end{equation}
This general initial condition, known as ``rough initial condition'', was introduced in \cite{CD2015}, and has been considered in the analysis of the stochastic heat equation in several other papers; see e.g. \cite{BC2018, CH2019, CK2019}. Here $|u_0| := u_{0, +} + u_{0, -}$, where $u_0 = u_{0, +} - u_{0, -}$ is the Jordan decomposition and $u_{0, \pm}$ are two non-negative Borel measure with disjoint support.

We assume that $W$ is a space-time homogeneous Gaussian noise. More precisely, $W=\{W(\varphi);\varphi \in C_0^{\infty}(\bR_{+} \times \bR^d)\}$ is a zero-mean Gaussian process, defined on a complete probability space $(\Omega,\cF,\bP)$, with the following covariance structure:
\begin{equation}
\label{cov}
\bE[W(\varphi)W(\psi)]=\int_{\bR_{+}^2 \times \bR^d} \cF \varphi(t,\cdot)(\xi) \overline{\cF \psi(s,\cdot)(\xi)}\gamma_0(t-s)dtds \mu(d\xi)=:\langle \varphi,\psi \rangle_{\cH},
\end{equation}
where $\cF \varphi(t,\cdot)(\xi)=\int_{\bR^d}e^{-i \xi \cdot x}\varphi(t,x)dx$ is the Fourier transform of the function $\varphi(t,\cdot)$, $\xi \cdot x$ is the Euclidean product in $\bR^d$, and $\gamma_0$ is a non-negative and non-negative definite, i.e. for any $\phi \in \mathcal{S}(\bR)$, 
$$\int_{\bR} (\phi \ast \widetilde{\phi})(t) \gamma_0(t) dt \ge 0.$$
We denote by $\mathcal{H}$ the Hilbert space defined as the completion of $C_0^{\infty}(\bR_{+} \times \bR^d)\}$ with respect to the inner product $\langle \cdot,\cdot \rangle_{\cH}$. In this setting, the noise $W$ can be extended to an isonormal Gaussian process $\{ W(\varphi) ; \varphi \in \cH \}$, and the solution can be defined using Malliavin calculus. 

By the Bochner-Schwartz theorem, there exists a tempered measure $\mu_0$ on $\bR$ such that $\gamma_0$ is the Fourier transform of $\mu_0$ in the space $\cS_{\bC}'(\bR)$ of tempered distributions on $\bR$. Then,
\begin{equation}
\label{Fourier-phi}
\int_{\bR^2}\phi(t)\phi(s)\gamma_0(t-s)dtds =\int_{\bR}|\cF \phi(\tau)|^2 \mu_0(d\tau),
\end{equation}
for any $\phi \in \cS_{\bC}(\bR)$, where $\cS_{\bC}(\bR)$ is the space of $\bC$-valued rapidly decreasing $C^{\infty}$-functions on $\bR$, and $\cF \phi$ is the Fourier transform of $\phi$, given by $\cF \phi(\tau)=\int_{\bR}e^{-i \tau t}\phi(t)dt$ for all $\tau \in \bR$. We assume that
there exists a function $g_0:\bR \to [0,\infty]$ such that
\begin{equation}
\label{cond-mu0}
\mu_0(d\tau)=g_0(\tau) d\tau.
\end{equation}
We will be interested in the temporal covariance function  
\begin{equation}
\label{temporal-cov}
\gamma_0(t) = \alpha_{H_0} |t|^{2H_0 - 2}
\end{equation}
with $H_0 \in (1/2, 1)$, in which case $g_0(\tau)=c_{H_0}|\tau|^{1-2H_0}$, with $\alpha_{H_0} = H_0(2H_0 - 1)$ and  
\begin{equation}
\label{constant_cH}
c_{H_0}=\frac{\Gamma(2H_0+1)\sin(\pi H_0)}{2\pi}.
\end{equation}
In this case, $\gamma_0$, we say that the noise $W$ is fractional in time with index $H_0 > 1/2$. 

Similarly, there exists a tempered measure $\mu$ on $\bR^d$ such that $\gamma$ is the Fourier transform of $\mu$ in $\cS'_{\bC}(\bR)$, i.e. for all $\varphi \in \cS_{\bC}(\bR^d)$, 
\begin{equation}
\label{Fourier-varphi}
\int_{\bR^{2d}}\varphi(x) \varphi(y)\gamma(x-y)dxdy =\int_{\bR^d}|\cF \varphi(\xi)|^2 \mu(d\xi).
\end{equation}
The spatial correlation pattern of the noise is defined by this spatial spectral measure $\mu$. Typically, this measure is assumed to be tempered, non-negative, and non-negative definite. We will consider two cases:
\[
\mu(d\xi)=
\left\{
\begin{array}{ll}
|\xi|^{-\alpha}d\xi  & \mbox{with $\alpha \in (0,d)$ and $d\geq 1$ ({\em Case I: the regular case})} \\
c_H|\xi|^{1-2H}d\xi & \mbox{with $H \in (0,1/2)$ and $d=1$ ({\em Case II: the rough case})}
\end{array} \right.
\]
In Case II, the constant $c_H$ is defined in \eqref{constant_cH}. 

We are now ready to give the definition of the solution to equation \eqref{PAM}.
\begin{definition}
We say that a process $u = \{u(t,x), (t,x) \in \bR_{+} \times \bR^d \}$ is a {\bf (mild Skorohod) solution} of equation \eqref{PAM} if for any $(t,x) \in \bR_{+} \times \bR^d$, $u(t,x)$ is $\mathcal{F}_t$-measurable, $\mathbb{E}|u(t,x)|^2 < \infty$ and solves the following equation:
$$u(t,x)=w(t,x)+\int_0^t \int_{\bR^d}G_{t-s}(x-y)u(s,y) W(\delta s, \delta y),$$
where the stochastic integral is interpreted in the Skorohod sense and $w(t,x)$ is the solution of the homogeneous heat equation with initial data $u_0$:
\begin{equation}
\label{defn-of-w}
w(t,x) = \int_{\bR^d} G_t(x-y) u_0(\mathrm{d}y).
\end{equation}
\end{definition}
Note that if $u_0 \ge 0$, then \eqref{eqn1-20230522-11:45am} is the necessary and sufficient condition for $w(t,x)$ to be finite. By Lemma B.2 of \cite{BC2018}, we know that $w(t,x)$ is continuous on $\bR_{+} \times \bR^d$. Hence $w(t,x)$ is uniformly bounded in any compact subset of $\bR_{+} \times \bR^d$. We recall that $G_t(x)$ is the fundamental solution of the heat equation in $\bR^d$ with
$$G_t(x)=\frac{1}{(2\pi t)^{d/2}}\exp\left(-\frac{|x|^2}{2t} \right),$$
and the Fourier transform of $G_t$ in any dimension $d$ is given by $\cF G_t(\xi)= e^{-t|\xi|^2/2}$. Moreover, we note that $|w(t,x)| \le w_{+}(t,x)$, where 
\begin{equation}
w_{+}(t,x) =  \int_{\bR^d} G_t(x-y) |u_0|(\mathrm{d}y),
\end{equation}
and it can be shown that condition \eqref{eqn1-20230522-11:45am} is equivalent to
$$w_{+}(t,x) < \infty \quad \text{ for all } t > 0 \text{ and } x \in \bR^d.$$

The solution to equation \eqref{PAM} is understood in {\bf Skorohod} sense. We recall that $\delta : \text{Dom}(\delta) \subset L^2(\Omega; \cH) \to L^2(\Omega)$ is the divergence operator with respect to $W$ which is defined as the adjoint of the Malliavin derivative $D$ with respect to $W$. If $u \in \text{Dom}(\delta)$, we use the notation
$$\delta(u) = \int_{0}^{\infty} \int_{\bR^d} u(t,x) W(\delta t, \delta x)$$
and we say $\delta(u)$ is the {\em Skorohod integral} of $u$ with respect to $W$. Further elaboration on the subject of Malliavin calculus is available in \cite{nualart06}. 

Using the methodology introduced in \cite{hu-nualart09}, we know that if it exists, the solution of equation \eqref{PAM} has the series expansion:
\begin{equation}
\label{series}
u(t,x)=w(t,x)+\sum_{n\geq 1}I_n\big(f_{t,x,n}\big) \quad \mbox{in} \quad L^2(\Omega),
\end{equation}
the terms of these series being orthogonal in $L^2(\Omega)$.
Here $I_n$ is the multiple Wiener integral of order $n$ with respect to $W$, and $f_{t,x,n}$ is given by: for all $0<t_1<\ldots<t_n<t$,
\begin{align*}
f_{t,x,n}(t_1,x_1,\ldots,t_n,x_n)&=G_{t-t_n}(x-x_n)\ldots G_{t_2-t_1}(x_2-x_1) w(t_1, x_1)\\
&= \int_{\bR^d} G_{t-t_n}(x - x_n) \cdots G_{t_2-t_1}(x_2 - x_1) G_{t_1}(x_1-x_0) u_0(\mathrm{d}x_0)
\end{align*}
where the second expression of this kernel is obtained by \eqref{defn-of-w}. By the orthogonality of the terms in this series, the necessary and sufficient condition for the existence of solution is
\[
\sum_{n\geq 1}\bE|I_n\big(f_{t,x,n}\big) |^2 =\sum_{n\geq 1}n! \, \|\widetilde{f}_{t,x,n}\|_{\cH^{\otimes n}}^2<\infty.
\]
where
$\widetilde{f}$ is the symmetrization of the function $f$, defined by:
\[
\widetilde{f}(t_1,x_1,\ldots,t_n,x_n)=\frac{1}{n!}\sum_{\rho \in S_n}f(t_{\rho(1)},x_{\rho(1)},\ldots,t_{\rho(n)},x_{\rho(n)}),
\]
and $S_n$ is the set of permutations of $\{1,\ldots,n\}$. 

In what follows, we will need the following result which gives the explicit expression of the Fourier transform of the kernel $f_{t,x,n}$ in the space variables.  
\begin{lemma}[Lemma 2.5 of \cite{BC2018}]
\label{lem1-20230317-4:21pm}
For any $0 < t_1 < \ldots < t_n < t = t_{n+1}$ and for any $\xi_1, \ldots, \xi_n \in \bR^d$, we have 
\begin{align*}
&\mathcal{F}f_{t,x,n}(t_1, \cdot,\ldots, t_n, \cdot)(\xi_1, \ldots, \xi_n)\\
&\qquad = \prod_{k = 1}^{n} \exp \bigg\{ -\frac{1}{2} \frac{t_{k+1} - t_k}{t_k t_{k+1}} \Big| \sum_{j = 1}^{k} t_j \xi_j \Big|^2 \bigg\}
\exp \bigg\{ -\frac{i}{t} \Big( \sum_{j = 1}^{n} t_j \xi_j  \Big) \cdot x \bigg\}\\
&\qquad \quad \times \int_{\bR^d} \exp\Big\{ -i \Big[ \sum_{j = 1}^{n} \Big( 1 - \frac{t_j}{t} \Big) \xi_j \Big] \cdot x_0  \Big\} G_t(x-x_0) u_0(\mathrm{d}x_0).
\end{align*}
\end{lemma}

In Case I with $\gamma_0$ arbitrary, a sufficient condition for the existence of the solution the heat equation \eqref{PAM} is {\em Dalang's condition}:
\[
\int_{\bR^d} \frac{1}{1+|\xi|^2}\mu(d\xi)<\infty,
\]
which is equivalent to
\begin{equation}
\label{dalang-cond}
d-\alpha<2.
\end{equation}
(see \cite{BC14}). In Case II with temporal covariance $\gamma_0$ given by \eqref{temporal-cov}, a sufficient condition for the existence of the solution is $H_0+H>3/4$ (see \cite{hu-le19}). We note that the authors of \cite{hu-le19} considered a slightly different initial condition compared to the present paper. 

%Due to the rough initial condition setting, the kernel $f_{t,x,n}$ has a different expression and requires different analysis 

\medskip

% Main results

The following theorems are the main results of the present article.

\begin{theorem}[The Regular Case]
\label{main-th1}
Let $\{W^{\alpha}\}_{\alpha \in (0,d)}$ be a family of zero-mean Gaussian processes with covariance \eqref{cov} in which $\gamma_0$ is chosen such that its corresponding measure $\mu_0$ satisfies \eqref{cond-mu0}, and the measure $\mu$ is given by:
\begin{equation}
\label{def-mu1}
\mu(d\xi)=|\xi|^{-\alpha}d\xi \quad \mbox{with $\alpha \in (0,d)$ and $d\geq 1$}.
\end{equation}
Let $u^{\alpha}$ be a continuous modification of the Skorohod solution of equation \eqref{PAM} with noise $W$ replaced by $W^{\alpha}$. Assume that the initial condition $u_0$ is a signed measure which satisfies \eqref{eqn1-20230522-11:45am}. For any $\alpha \in \big(\max(d-2, 0),d \big)$ and for any $t_0 \in (0,T)$, if $\alpha_n \to \alpha^* \in (\max(d-2,0),d)$, then
\[
u^{\alpha_n} \stackrel{d}{\longrightarrow} u^{\alpha^*} \quad \
\mbox{in $C([t_0,T] \times \bR^d)$},
\]
where $C([t_0,T] \times \bR^d)$ is equipped with the topology of uniform convergence on compact sets.
\end{theorem}

\begin{theorem}[The Rough Case]
\label{main-th2}
Let $\{W^{H} \}_{H \in (0,1/2)}$ be a family of zero-mean Gaussian processes with covariance \eqref{cov} in which $\gamma_0$ is given by \eqref{temporal-cov} and the measure $\mu$ is given by:
\begin{equation}
\label{def-mu2}
\mu(d\xi)= c_H |\xi|^{1-2H}d\xi \quad \mbox{with $H \in (0,1/2)$ and $d=1$.}
\end{equation}
Fix $H_0 \in (1/2,1)$. We define $\ell= \max(3/4-H_0,0)$. Assume that the initial condition $u_0$ is a signed measure which satisfies \eqref{eqn1-20230522-11:45am}. For any $H \in (\ell,1/2)$ and for any $t_0 \in (0,T)$, let $u^{H}$ be a continuous modification of the Skorohod solution of equation \eqref{PAM} with noise $W$ replaced by $W^{H}$.
If $H_n \to H^* \in (\ell,1/2)$, then
\[
u^{H_n} \stackrel{d}{\longrightarrow} u^{H^*} \quad \
\mbox{in $C([t_0,T] \times \bR)$},
\]
where $C([t_0,T] \times \bR)$ is equipped with the topology of uniform convergence on compact sets.
\end{theorem}

\begin{remark}
We note that in Theorems \ref{main-th1} and \ref{main-th2}, the convergence is only on compact sets of the form $[t_0, T] \times K$ with $0 < t_0 < T$ and $K \subset \bR^d$ compact. This limitation is due to some uniform moment estimates, which are needed for the proof of tightness. 
\end{remark}

%Whenever the while family of noise $\{ W^{\alpha},  \alpha \in (0,d)\}$ (or $\{ W^{H}, H  \in (0,1/2)\}$) are defined in a single probability space, we check that for any fixed $(t,x) \in [0,T] \times \bR^d$, $u^{\alpha_n}(t,x)$ (respectively $u^{H_n}(t,x)$) converge to  $u^{\alpha^\ast}(t,x)$ (respectively $u^{H^\ast}(t,x)$) in $L^2(\Omega)$. For tightness, we show that the sequence of probability measures induced by $\{ u^{\alpha_n}, n \ge 1 \}$, (or $\{ u^{H_n}, n \ge 1 \}$) is tight in the space $C([t_0,T] \times \bR^d)$ (or $C([t_0,T] \times \bR)$). 

The proof of Theorems \ref{main-th1} and \ref{main-th2} follow the same strategy. We first prove the convergence of the corresponding finite-dimensional distributions and tightness. In both cases, we need to guarantee that all noise processes are defined on the same probability space, for all parameter values $\alpha$, or $H$. To construct all noise processes $(W^{\alpha})_{\alpha \in (0,d)}$, respectively $(W^{H})_{H \in (0,1/2)}$, on the same probability space, we need to assume that $\mu_0$ has density (see \eqref{cond-mu0}). We refer reader to \cite{BL2023} for this construction of noise, based on a complex-valued Gaussian white noise. To prove tightness, we apply Kolmogorov-Centsov theorem. For this, we derive some moment estimates for the increments of the solution, ensuring that all bounds remain uniformly bounded for parameter values ($\alpha$ or $H$) in a compact set.

We conclude the introduction by providing a brief overview of the article's organization and introducing the notation that will be used throughout. The proofs of Theorems \ref{main-th1} and \ref{main-th2} can be found in Section \ref{Regular-noise}, respectively Section \ref{Rough-noise}. We denote
\[
f(\pmb{t_n},\pmb{x_n})=f(t_1,x_1,\ldots,t_n,x_n),
\]
where $\pmb{t_n}=(t_1,\ldots,t_n)\in \bR_{+}^n$ and $\pmb{x_n}=(x_1,\ldots,x_n)\in (\bR^d)^n$. Occasionally, we will use $\pmb{\xi_n}=(\xi_1,\ldots,\xi_n) \in (\bR^d)^n$. We will use the convention:
\begin{equation}
\label{convention}
G_t(x)=0 \quad \mbox{for any $t<0$ and $x \in \bR^d$.}
\end{equation}
We let $T_n(t)=\{\pmb{t_n}=(t_1,\ldots,t_n); 0<t_1<\ldots<t_n<t\}$ be the $n$-dimensional simplex. For any $p\geq 1$, we denote by $\|\cdot \|_p$ the norm in $L^p(\Omega)$.
We let $\sigma$ be the surface measure on the unit sphere $S_1(0)=\{z\in \bR^d;|z|=1\}$, and $c_d$ be the area of $S_1(0)$, i.e.
\begin{equation}
\label{def-cd}
c_d=\int_{S_1(0)}\sigma(dz).
\end{equation}
%%%%%%%%%%%%%%%%%%%%%%%%%

%%%%%%%%%%%%%%%%%%%%%%%%%%%%%%%%%%%%%%%%%%%%%%%%

%%%%%%%%%%%%%%%%%%%%%%%%%%%%%%%%%%%%%%%%%%%%%%%%
\section{Regular noise}
\label{Regular-noise}

In this section, we consider equation \eqref{PAM} driven by a Gaussian noise $W$ with covariance \eqref{cov} in which $\gamma_0$ is chosen such that its corresponding measure $\mu_0$ satisfies \eqref{cond-mu0}, and the measure $\mu$ is given by \eqref{def-mu1}. The existence and uniqueness of the solution was proved in \cite{BC2018}.

To emphasize the dependence on the parameter $\alpha$, we denote the noise, the Hilbert space, and the solution, by $W^{\alpha},\cH^{\alpha},u^{\alpha}$, respectively. Denoting the multiple integral of order $n$ with respect to $W^{\alpha}$ as $I_n^{\alpha}$, the series expansion \eqref{series} takes the following form:
\[
u^{\alpha}(t,x)=1+\sum_{n\geq 1}I_{n}^{\alpha}(f_{t,x,n}).
\]
As mentioned above, we need the family $\{ W^{\alpha} \}_{\alpha \in (0,d)}$ to be defined on the same probability space. We omit the details and refer the reader to \cite{BL2023} for this constructions, which involves a $\bC$-valued Gaussian measure $\widehat{W}$, defined as follows:
\begin{equation}
\label{def-W-hat}
\widehat{W}(A)=W_1(A)+iW_2(A)
\end{equation}
where $W_1$ and $W_2$ are independent space-time Gaussian white noise processes on $\bR^{d+1}$.
We recall the following representation of the multiple integral with respect to the process $W^{\alpha}$: for any $t \mapsto \varphi(t, \cdot) \in \mathcal{S}'(\bR^d)$,
\begin{equation}
\label{In-alpha}
I_k^{\alpha}(\varphi)=\int_{(\bR \times \bR^d)^k} \cF_t[\cF_x \varphi(\pmb{t_k},\bullet)(\pmb{\xi_k})](\pmb{\tau_k})
\prod_{j=1}^{k}\sqrt{g_0(\tau_j)}|\xi_j|^{-\alpha/2}
\widehat{W}(d\tau_1,d\xi_1) \ldots \widehat{W}(d\tau_k,d\xi_k),
\end{equation}
where $\cF_x$ is the Fourier transform in the space variables and $\cF_t $ is the Fourier transform in the time variables, i.e. the Fourier transform of the function $\pmb{t_k} \mapsto \cF_x \varphi(\pmb{t_k},\bullet)(\pmb{\xi_k})$. Whenever there is no risk of confusion, we drop the lower indices $t,x$ from the Fourier transform notation.

The following three lemmas will play an important role below. The first one can be found for instance in \cite{B12-POTA} and the rest are taken from \cite{BC2018}.

\begin{lemma}
\label{G-lemma}
For any $\alpha \in \big(\max(d-2,0),d\big)$ and $t>0$,
\begin{align*}
& \int_{\bR^d}|\cF G_t(\xi)|^2 |\xi-\eta|^{-\alpha}d\xi \leq K_{d,\alpha} \,t^{-(d-\alpha)/2},
\end{align*}
where
\begin{equation}
\label{def-K}
K_{d,\alpha} := \int_{\bR^d}\frac{1}{1+|\xi|^{2}}|\xi|^{-\alpha}d\xi \leq c_{d}\left( \frac{1}{d-\alpha}+\frac{1}{2-(d-\alpha)}\right),
\end{equation}
with $c_d$ is given by \eqref{def-cd}. 
%\begin{equation}
%\label{def-r}
%r_{\alpha}=
%\left\{
%\begin{array}{ll}
%-(d-\alpha)/2 & \mbox{for the heat equation} \\
%2-(d-\alpha) & \mbox{for the wave equation}
%\end{array} \right.
%\end{equation}
\end{lemma}

\begin{lemma}[Lemma 3.4 of \cite{BC2018}]
\label{lem1-20230321-4:23pm}
Let $\mu$ be a tempered measure on $\bR^d$ such that its Fourier transform in $\mathcal{S}'_{\mathbb{C}}(\bR^d)$ is a locally integrable function $\gamma$. Assume that $\gamma$ is non-negative. Then for any $\psi \in \mathcal{S}(\bR^d)$ such that $\psi \ast \widetilde{\psi}$ is non-negative, where $\widetilde{\psi}(x) = \psi(-x)$ for all $x \in \bR^d$, we have:
\begin{align*}
\sup_{\eta \in \bR^d} \int_{\bR^d} \big| \mathcal{F} \psi(\xi + \eta) \big|^2 \mu(\mathrm{d}\xi)
= \int_{\bR^d} \big| \mathcal{F} \psi(\xi) \big|^2 \mu(\mathrm{d}\xi).
\stepcounter{equation}\tag{\theequation}\label{relation2-20230322-2:19pm} 
\end{align*}
In particular, for any $a > 0$ and $t > 0$,
\begin{align*}
\sup_{\eta \in \bR^d} \int_{\bR^d} e^{-a|t\xi + \eta|^2} \mu(\mathrm{d}\xi)
= \int_{\bR^d} e^{-a|t\xi|^2} \mu(\mathrm{d}\xi).
\stepcounter{equation}\tag{\theequation}\label{relation3-20230322-2:19pm} 
\end{align*}
\end{lemma}

\begin{lemma}[Lemma 4.1 of \cite{BC2018}]
\label{lem1-20230328-3:39pm}
For any $h > -1$, we have
$$\int_{0 < t_1 < \ldots < t_n < t} \big[ t_1(t_2 - t_1) \ldots (t - t_n) \big]^h \mathrm{d}t_1 \ldots \mathrm{d}t_n = \frac{\Gamma(h+1)^{n+1}}{\Gamma\Big((n+1)(h+1) \Big)}t^{n(h+1)+h}.$$
\end{lemma}

% Lemma for finite dimensional distribution convergence 

The following result will be used in the proof of Theorem \ref{main-th1}, in order to show the finite dimensional convergence.

\begin{lemma}
\label{Ik-conv}
If $\alpha_n \to \alpha^* \in \big(\max(d-2,0),d\big)$, then for any $t>0$, $x \in \bR^d$ and $k\geq 1$,
\[
\bE|I_k^{\alpha_n}(f_{t,x,k}) - I_k^{\alpha^*}(f_{t,x,k})|^2 \to 0, \quad \mbox{as} \quad n\to \infty.
\]
\end{lemma}

\begin{proof}
By \eqref{In-alpha}, we have
\begin{align*}
I_k^{\alpha_n}(f_{t,x,k}) - I_k^{\alpha^*}(f_{t,x,k})
=\int_{(\bR \times \bR^d)^n} H_k^{(n)}(\pmb{\tau_k},\pmb{\xi_k})
\widehat{W}(d\tau_1,d\xi_1) \ldots \widehat{W}(d\tau_n,d\xi_n),
\stepcounter{equation}\tag{\theequation}\label{relation3-20230728-7:46pm} 
\end{align*}
where
\[
H_k^{(n)}(\pmb{\tau_k},\pmb{\xi_k})=\cF \phi_{\pmb{\xi_k}}(\pmb{\tau_k})
\prod_{j=1}^{k}\sqrt{g_0(\tau_j)}
\left(\prod_{j=1}^k|\xi_j|^{-\alpha_n/2}-
\prod_{j=1}^k|\xi_j|^{-\alpha^*/2}\right)
\]
and
\begin{align*}
\phi_{\pmb{\xi_k}}(\pmb{t_k})&:=\cF f_{t,x,k}(\pmb{t_k},\cdot)(\pmb{\xi_k})\\
&= \prod_{j = 1}^{k} \exp \bigg\{ -\frac{1}{2} \frac{t_{j+1} - t_j}{t_j t_{j+1}} \Big| \sum_{i = 1}^{j} t_i \xi_i \Big|^2 \bigg\}
\exp \bigg\{ -\frac{i}{t} \Big( \sum_{j = 1}^{k} t_j \xi_j  \Big) \cdot x \bigg\}\\
&\qquad \quad \times \int_{\bR^d} \exp\Big\{ -i \Big[ \sum_{j = 1}^{k} \Big( 1 - \frac{t_j}{t} \Big) \xi_j \Big] \cdot x_0  \Big\} G_t(x-x_0) u_0(\mathrm{d}x_0),
\stepcounter{equation}\tag{\theequation}\label{def-phi-k} 
\end{align*}
with $t_{k+1}=t$. 
Note that \eqref{relation3-20230728-7:46pm} is the same as in the proof of Lemma 3.3 of \cite{BL2023}, but here we use a different expression for the Fourier transform $\cF f_{t,x,k}(\pmb{t_k},\cdot)(\pmb{\xi_k})$, due to the general initial condition. This form is specific to the heat equation, since it was developed using the semigroup property of the heat kernel. Proceeding as in the proof of Lemma 3.3 of \cite{BL2023}, we infer that
\begin{align*}
Q_n :=\bE|I_k^{\alpha_n}(f_{t,x,k}) -I_k^{\alpha^*}(f_{t,x,k})|^2 \le k! \Gamma_{0,t}^k\int_{T_k(t)} A_k^{(n)}(\pmb{t_k},\pmb{t_k}) d \pmb{t_k}
\end{align*}
where $T_k(t)=\{0<t_1<\ldots<t_k<t\}$, 
\begin{align*}
A_k^{(n)}(\pmb{t_k},\pmb{s_k})=\int_{(\bR^d)^k}
 \phi_{\pmb{\xi_k}}(\pmb{t_k})  \phi_{\pmb{\xi_k}}(\pmb{s_k})  \bigg|\prod_{j=1}^k|\xi_j|^{-\alpha_n/2}-
\prod_{j=1}^k|\xi_j|^{-\alpha^*/2}\bigg|^2 d\pmb{\xi_k}
\end{align*}
and $\Gamma_{0,t}=\int_{-t}^t \gamma_0(s)ds$. It suffices to show that the integral appearing in the upper bound above converges to $0$ as $n\to \infty$, i.e.
$$T_n :=
\int_{T_k(t)} \int_{(\bR^d)^k} \big| \phi_{\pmb{\xi_k}}(\pmb{t_k}) \big|^2 \left|\prod_{j=1}^k|\xi_j|^{-\alpha_n/2}-
\prod_{j=1}^k|\xi_j|^{-\alpha^*/2}\right|^2
  d\pmb{\xi_k} d \pmb{t_k} \to 0.
$$
Note that the integrand converges pointwise to $0$ on $T_k(t) \times (\bR^d)^k$, as $n \to \infty$.
To apply the Dominated Convergence Theorem, it suffices to find an integrable bound for the term depending on $\alpha_n$:
\[
g_n(\pmb{t_k},\pmb{\xi_k}):= \big| \phi_{\pmb{\xi_k}}(\pmb{t_k}) \big|^2 \prod_{j=1}^k|\xi_j|^{-\alpha_n},
\]
and we need to find a function $g$ on $T_k(t) \times (\bR^d)^k$, such that
\[
g_n(\pmb{t_k},\pmb{\xi_k}) \leq g(\pmb{t_k},\pmb{\xi_k}) \quad \mbox{for all $n$}, \quad \mbox{and} \quad \int_{T_k(t)} \int_{(\bR^d)^k}g(\pmb{t_k},\pmb{\xi_k}) d\pmb{\xi_k} d\pmb{t_k}<\infty.
\]

Recall that $\max\{ d-2, 0 \} < \alpha^{\ast} < d$. Fix numbers $a$ and $b$ such that $\max\{ d-2, 0 \} < a < \alpha^{\ast} < b < d$. Since $\alpha_n \to \alpha^{\ast}$, there exists $N \in \mathbb{N}$ such that
$a \le \alpha_n \le b$ for all $n \ge N$. A natural candidate for $g$ is
\begin{align*}
g(\pmb{t_k},\pmb{\xi_k}) &:= \big| \phi_{\pmb{\xi_k}}(\pmb{t_k}) \big|^2
\prod_{j=1}^k  \big(|\xi_j|^{-b}1_{\{ |\xi_j| \le 1 \}}  + |\xi_j|^{-a}1_{\{|\xi_j| > 1 \}}\big)\\
&\le w_{+}^2(t,x) \prod_{j = 1}^{k} \exp \bigg\{ -\frac{t_{j+1} - t_j}{t_j t_{j +1}} \Big| \sum_{i = 1}^{j} t_i \xi_i \Big|^2 \bigg\} \Big( |\xi_j|^{-b}1_{\{ |\xi_j| \le 1 \}}  + |\xi_j|^{-a} 1_{\{ |\xi_j| > 1 \}}\Big).
\stepcounter{equation}\tag{\theequation}\label{relation4-20230728-8:01pm} 
\end{align*}
We fix $j = 1, \ldots, k$. By Lemma \ref{G-lemma}, we have
\begin{align*}
&\int_{\bR^d}  e^{-\frac{t_{j+1} - t_j}{t_j t_{j +1}} \big|\sum_{i = 1}^{j} t_i \xi_i \big|^2} \Big( |\xi_j|^{-b}1_{\{ |\xi_j| \le 1 \}}  + |\xi_j|^{-a} 1_{\{ |\xi_j| > 1 \}}\Big)  \mathrm{d}\xi_j\\
&\qquad \le \int_{|\xi_j| \le 1} |\xi_j|^{-b} \mathrm{d}\xi_j + \sup_{\eta \in \bR^d} \int_{\bR^d} e^{-\frac{t_{j+1} - t_j}{t_j t_{j +1}} \big| t_j \xi_j + \eta \big|^2} |\xi_j|^{-a}\mathrm{d}\xi_j\\
&\qquad = \frac{c_d}{d - b} + \int_{\bR^d}  e^{-\frac{t_{j+1} - t_j}{t_j t_{j +1}} | t_j \xi_j |^2} |\xi_j|^{-a}\mathrm{d}\xi_j = \frac{c_d}{d - b} + \int_{\bR^d}  e^{-\frac{(t_{j+1} - t_j)t_j }{t_{j +1}} \xi_j^2} |\xi_j|^{-a}\mathrm{d}\xi_j\\
&\qquad \le \frac{c_d}{d - b} + \bigg( \frac{(t_{j+1} - t_j)t_j }{t_{j +1}} \bigg)^{-\frac{d - a}{2}} K_{d, a} 
=: C\Big[ 1 + \bigg( \frac{(t_{j+1} - t_j)t_j }{t_{j +1}} \bigg)^{-\frac{d - a}{2}} \Big]
\stepcounter{equation}\tag{\theequation}\label{relation1-20230404-3:19pm} 
\end{align*}
where $c_d$ given in \eqref{def-cd} and $K_{d, a}$ are given in \eqref{def-K}. Observe that 
for any $0 < t_1 < \ldots < t_k < t_{k+1} = t$, we have $\frac{(t_{j+1} - t_j)t_j }{t_{j +1}} \le t_{j+1} - t_j \le t$. Therefore we have $\Big( \frac{(t_{j+1} - t_j)t_j }{t_{j +1}} \Big)^{-\frac{d - a}{2}} \ge t^{-\frac{d - a}{2}},$
and multiplying by $t^{\frac{d - a}{2}}$ on both sides, we obtain:
$1 \le t^{\frac{d - a}{2}} \Big( \frac{(t_{j+1} - t_j)t_j }{t_{j +1}} \Big)^{-\frac{d - a}{2}}  $
which implies
\begin{align*}
1 + \Big( \frac{(t_{j+1} - t_j)t_j }{t_{j +1}} \Big)^{-\frac{d - a}{2}} &\le \Big( \frac{(t_{j+1} - t_j)t_j }{t_{j +1}} \Big)^{-\frac{d - a}{2}} \Big( t^{\frac{d - a}{2}}  + 1 \Big).
\stepcounter{equation}\tag{\theequation}\label{relation1-20230728-10:01am} 
\end{align*}
Using relations \eqref{relation4-20230728-8:01pm}, \eqref{relation1-20230404-3:19pm} and \eqref{relation1-20230728-10:01am}, we have
\begin{align}
\int_{T_k(t)} \int_{(\bR^d)^k} g(\pmb{t_k},\pmb{\xi_k}) d\pmb{\xi_k} d \pmb{t_k} &\le w^2(t,x) C^k \int_{T_k(t)} \prod_{j=1}^{k} \Big( \frac{(t_{j+1} - t_j)t_j }{t_{j +1}} \Big)^{-\frac{d - a}{2}} \mathrm{d} \bold{t} \nonumber \\
&\le C \frac{\Gamma\big(\frac{2-(d-a)}{2}\big)^{n+1}}{\Gamma\Big((n+1)(\frac{2-(d-a)}{2}) \Big)}t^{n(\frac{2-(d-a)}{2})-\frac{d-a}{2}} \nonumber
\end{align}
where for the last line, we used Lemma \ref{lem1-20230328-3:39pm}. 
\end{proof}

% uniform moment estimates

For the proof of tightness, we need some upper bounds for the moments of the increments of $u_{\alpha}$. These bounds have been obtained in \cite{BC2018} and are given by Theorem \ref{Holder-th} below. However, for our problem, we will need the explicit form of the constant $C$ appearing in this theorem, and this will be described in Theorem \ref{Th-uniform-C}. Note that Theorem \ref{Holder-th} guarantees the existence of a H\"{o}lder continuous modification of the solution $u_{\alpha}$. We work with this modification, which we denote also by $u_{\alpha}$.  

\begin{theorem}[Theorem 1.3 of \cite{BC2018}]
\label{Holder-th}
Let $u^{\alpha}$ be the solution of equation \eqref{PAM} starting from an initial measure $u_0$ that satisfying \eqref{eqn1-20230522-11:45am}. Suppose that:
$$\int_{\bR^d} \Big( \frac{1}{1 + |\xi|^2} \Big)^{\beta} \mu(d \xi) < \infty \quad \text{ for some } \beta \in (0,1).$$
Then for any $p \ge 2$ and $a > 1$ there exists a constant $C > 0$ depending on $p, a, \lambda$ and $\beta$ such that for any $(t,x)$, $(t',x') \in K_a := [1/a, a] \times [-a, a]^d$, 
\begin{align}
\label{Holder-eq}
& \|u^{\alpha}(t,x)-u^{\alpha}(t',x')\|_p \leq C\big(|t-t'|^{\frac{1-\beta}{2}} +|x-x'|^{1-\beta}\big),
\end{align}
Consequently, for any $a > 1$, the process $\{u^{\alpha}(t,x); (t,x) \in K_a\}$ has a modification which is a.s. $\theta_1$-H\"{o}lder continuous in time and a.s. $\theta_2$-H\"{o}lder continuous in space, for any $\theta_1 \in (0, (1-\beta)/2)$ and $\theta_2 \in (0, 1-\beta)$.
\end{theorem}

We will use the following constant:
\begin{equation}
\label{k_alpha_constant}
k_{\alpha}(t) :=  \int_{\bR^d} e^{-\frac{t|\xi|^2}{2}} |\xi|^{-\alpha} d \xi = C^{(1)}_{d, \alpha} \, t^{-(d - \alpha)/2},
\end{equation}
where the last equation is obtained using the change of variables $\xi' = \sqrt{t}\xi$. The constant $C^{(1)}_{d, \alpha}$ is defined by:
$$C^{(1)}_{d, \alpha} := \int_{\bR^d} e^{-|\xi|^2/2}|\xi|^{-\alpha} d \xi.$$
Note that $C^{(1)}_{d, \alpha}$ can be uniformly bounded, for all $\alpha \in [a,b]$ such that $\max\{d-2, 0 \} < a < b< d$. 

%Note that $k_{\alpha}(t)$ is a non-increasing function. Moreover, $k_{\alpha}(t)$ is continuous on $(0, \infty)$, by the dominated convergence theorem and condition \eqref{dalang-condition}. 

The following result will be used several times in the proof of Theorem \ref{Th-uniform-C} below.

\begin{lemma}[Lemma 3.5 of \cite{BC2018}]
\label{lem1-20230324-4:38pm}
For any $0 < t_1 < \ldots < t_n < t := t_{n+1}$, we have that
$$\int_{\bR^{nd}} \exp\bigg\{ -\sum_{k = 1}^{n} \bigg( \frac{t_{k+1} - t_{k}}{t_{k+1}t_{k}} \bigg| \sum_{j=1}^{k} t_j\xi_j \bigg|^2 \bigg) \bigg\} \prod_{k=1}^{n}|\xi_k|^{-\alpha} d \xi_k \le \prod_{j = 1}^{n} k_{\alpha}\bigg( \frac{2(t_{j+1} - t_j)t_j}{t_{j+1}} \bigg).$$
\end{lemma}
\noindent By \eqref{k_alpha_constant}, we have:
\begin{align*}
\prod_{j = 1}^{n} k_{\alpha}\bigg( \frac{2(t_{j+1} - t_j)t_j}{t_{j+1}} \bigg) &= C^{(1)}_{d, \alpha} \Big( \frac{2(t_2 - t_1)t_1}{t_2}\Big)^{-(d-\alpha)/2}  \ldots C^{(1)}_{d, \alpha} \Big(  \frac{2(t - t_n)t_n}{t} \Big)^{-(d-\alpha)/2} \\
&= \big( C^{(1)}_{d, \alpha} \big)^n 2^{-(d-\alpha)n/2} t^{(d - \alpha)/2} \big[  t_1(t_2 - t_1) \ldots (t - t_n)  \big]^{-(d-\alpha)/2}.
\stepcounter{equation}\tag{\theequation}\label{relation1-20230328-3:19pm}
\end{align*}

The next result gives the explicit form of the constant $C$ as a function of $\alpha$, which allow us to bound it uniformly for all $\alpha$ in a compact interval $[a,b]$. Then we can use it to demonstrate the tightness of $(u^{\alpha})_{\alpha \in [a,b]}$.

% Theorem for the uniform moment estimates
\begin{theorem}
\label{Th-uniform-C}
Assume that $\max(d-2,0)<a<b<d$. We fix $T > t_0 > 0$. Let $K \subset \bR^d$ be a compact set. For any $p\geq 2$, $\beta \in \big(\frac{d-a}{2},1 \big)$, $t,t' \in [t_0,T]$ and $x,x' \in K$, we have:
\[
\sup_{\alpha \in [a,b]} \|u^{\alpha}(t,x)-u^{\alpha}(t',x')\|_p \leq C \Big(|t-t'|^{\frac{1-\beta}{2}}+|x-x'|^{1-\beta}\Big).
\]
Here $C$ is a constant that depends on $(d,p,T,t_0,\beta, a,b)$.
\end{theorem}

\begin{proof}
We proceed with a similar argument as in the proof of Theorem 1.3 of \cite{BC2018}, with the addition of an examination of a missing term that was overlooked in that particular proof. To simplify the writing, we let $\theta = 1 - \beta$. 

% time increments

{\bf Step 1 (time increments).} Let $t \in [t_0,T]$ and $h>0$ be such that $t+h \in [t_0,T]$. By hypercontractivity,
\begin{align*}
\big\|u^{\alpha}(t+h,x)-u^{\alpha}(t,x)\big\|_p & \leq \sum_{n\geq 1}(p-1)^{n/2}\big\|I_n^{\alpha}\big( f_{t+h,x,n}-f_{t,x,n}\big)\big\|_2 \\
& \leq \sum_{n\geq 1}(p-1)^{n/2}
\left(\frac{2}{n!} \Big(A_n^{\alpha}(t,h)+B_n^{\alpha}(t,h)\Big) \right)^{1/2},
\end{align*}
where
$A_n^{\alpha}(t,h)=(n!)^2\| \widetilde{f}_{t+h,x,n}1_{[0,t]^n}-\widetilde{f}_{t,x,n}
\|_{\cH_{\alpha}^{\otimes n}}^2$ and $B_n^{\alpha}(t,h)=(n!)^2\| \widetilde{f}_{t+h,x,n}1_{[0,t+h]^n -[0,t]^n}\|_{\cH_{\alpha}^{\otimes n}}^2$.

% A_n^{\alpha}(t,h)
We examine $A_n^{\alpha}(t,h)$. Note that
\begin{align*}
A_n^{\alpha}(t,h) \leq \Gamma_{0,t}^n \int_{[0,t]^n} \psi_{t,h,n}^{\alpha}(\pmb{t_n})d \pmb{t_n},
\stepcounter{equation}\tag{\theequation}\label{rel3-20230908-1:57pm} 
\end{align*}
where the function $\psi_{t,h,n}^{\alpha}(\pmb{t_n})$ is defined as follows: if $\pmb{t_{n}} \in [0,t]^n$ is fixed and $\rho$ is the permutation of $1,\ldots,n$ such that $t_{\rho(1)}<\ldots<t_{\rho(n)}$ with $t_{\rho(n+1)}=t$, and
\begin{align*}
\psi_{t,h,n}^{\alpha}(\pmb{t_n}) & : =
 \int_{(\bR^d)^n} \Big|\cF \Big(f_{t+h,x,n}(t_{\rho(1)}, \cdot, \ldots, t_{\rho(n)}, \cdot) - f_{t,x,n}(t_{\rho(1)}, \cdot, \ldots, t_{\rho(n)}, \cdot)\Big)(\pmb{\xi_n}) \Big|^2 
\prod_{j=1}^{n}|\xi_j|^{-\alpha}d\pmb{\xi_n}.
\end{align*}
Assume for simplicity that $t_1 < \ldots < t_n$. By Lemma \ref{lem1-20230317-4:21pm}, 
\begin{align*}
\Big| \cF \big(f_{t+h,x,n}(\pmb{t_n}, \cdot) - f_{t,x,n}(\pmb{t_n}, \cdot)\big)(\pmb{\xi_n}) \Big| = \prod_{k = 1}^{n-1} \exp \Big( -\frac{1}{2} \frac{t_{k+1} - t_{k}}{t_{k} t_{k+1}} \Big| \sum_{j = 1}^{k} t_{j} \xi_{j} \Big|^2 \Big)
\Big|I^{(1)}_1 I^{(2)}_1- I^{(1)}_2 I^{(2)}_2 \Big|
\end{align*}
where
\begin{align*}
I^{(1)}_1 &:= \exp \Big( -\frac{1}{2} \frac{t+h - t_{n}}{t_{n}(t+h)} \Big| \sum_{j = 1}^{n} t_{j} \xi_{j} \Big|^2 \Big) \\
I^{(2)}_1 
&:= \int_{\bR^d} \exp\bigg( -i \Big( \sum_{j = 1}^{n} \xi_{j} \Big) \cdot x_0  \bigg) \exp \bigg( -\frac{i}{t+h} \Big( \sum_{j = 1}^{n} t_{j} \xi_{j}\Big) \cdot (x - x_0) \bigg) \, G_{t+h}(x-x_0) u_0(\mathrm{d}x_0)\\
I^{(1)}_2 &:= \exp \Big( -\frac{1}{2} \frac{t  - t_{n}}{t_{n} t} \Big| \sum_{j = 1}^{n} t_{j} \xi_{j} \Big|^2 \Big)\\ 
I^{(2)}_2 
&:= \int_{\bR^d} \exp\Big( -i \Big( \sum_{j = 1}^{n} \xi_{j} \Big) \cdot x_0  \Big) \exp \bigg( -\frac{i}{t} \Big( \sum_{j = 1}^{n} t_{j} \xi_{j}\Big) \cdot (x - x_0) \bigg) \, G_{t}(x-x_0) u_0(\mathrm{d}x_0).
\end{align*}
Introducing the mixed term $I^{(1)}_1 I^{(2)}_2$, we get:
\begin{align*}
\Big| \cF \big(f_{t+h,x,n}(\pmb{t_n}, \cdot) - f_{t,x,n}(\pmb{t_n}, \cdot)\big)(\pmb{\xi_n}) \Big| \le K_1 + K_2
\stepcounter{equation}\tag{\theequation}\label{rel1-20231205-11:51pm} 
\end{align*}
where
\begin{align*}
K_1 &:= \prod_{k = 1}^{n-1} \exp \Big( -\frac{1}{2} \frac{t_{k+1} - t_{k}}{t_{k} t_{k+1}} \Big| \sum_{j = 1}^{k} t_{j} \xi_{j} \Big|^2 \Big) \times \Big( I^{(1)}_1 \big| I^{(2)}_1 - I^{(2)}_2 \big| \Big),  \\
K_2 &:= \prod_{k = 1}^{n-1} \exp \Big( -\frac{1}{2} \frac{t_{k+1} - t_{k}}{t_{k} t_{k+1}} \Big| \sum_{j = 1}^{k} t_{j} \xi_{j} \Big|^2 \Big) \times \Big( \big| I^{(2)}_2 \big| \big| I^{(1)}_1 - I^{(1)}_2 \big|  \Big).
\end{align*}
% K_1
We first consider $K_1$. Note that
\begin{align*}
\Big| I^{(2)}_1 - I^{(2)}_2 \Big|
&\le \int_{\bR^d} \bigg|
\exp \Big( -\frac{i}{t+h} \Big( \sum_{j = 1}^{n} t_{j} \xi_{j}\Big) \cdot (x - x_0) \Big) \, G_{t+h}(x-x_0)\\
&\qquad \qquad -
\exp \Big( -\frac{i}{t} \Big( \sum_{j = 1}^{n} t_{j} \xi_{j}\Big) \cdot (x - x_0) \Big) \, G_{t}(x-x_0) \bigg| u_0(\mathrm{d}x_0) \le T_1 + T_2
\end{align*}
where 
\begin{align*}
T_1 &:= \int_{\bR^d} \bigg|
\exp \Big( -\frac{i}{t+h} \Big( \sum_{j = 1}^{n} t_{j} \xi_{j}\Big) \cdot (x - x_0) \Big) \, G_{t+h}(x-x_0)\bigg|\\
&\qquad \qquad \qquad \qquad \qquad \qquad \qquad \bigg| G_{t+h}(x-x_0) - G_{t}(x-x_0) \bigg|u_0(\mathrm{d}x_0)\\
T_2 
&:=  \int_{\bR^d} \bigg| \exp \Big( -\frac{i}{t+h} \Big( \sum_{j = 1}^{n} t_{j} \xi_{j}\Big) \cdot (x - x_0) \Big) \\
&\qquad \qquad \qquad \qquad \qquad -  \exp \Big( -\frac{i}{t} \Big( \sum_{j = 1}^{n} t_{j} \xi_{j}\Big) \cdot (x - x_0) \Big) \bigg| G_{t}(x-x_0) u_0(\mathrm{d}x_0).
\end{align*}

We study $T_1$ first. Using the inequality $1 - e^{-x} \le x^{\theta}$ for any $x > 0$ and for any $\theta \in [0,1]$, we have:
\begin{align*}
T_1 
&\le \int_{\bR^d} \bigg| G_{t+h}(x-x_0) - G_{t}(x-x_0) \bigg|u_0(\mathrm{d}x_0) \le \frac{1}{(2 \pi t)^{d/2}}\int_{\bR^d} e^{- \frac{|x - x_0|^2}{2(t+h)}} \Big| 1 - e^{- \frac{h|x - x_0|^2}{2t(t+h)}} \Big| u_0(\mathrm{d}x_0)\\
&\le \frac{1}{(2 \pi t)^{d/2}} \int_{\bR^d} e^{- \frac{|x - x_0|^2}{2(t+h)}} \Big( \frac{h|x - x_0|^2}{2t(t+h)} \Big)^{\theta} u_0(\mathrm{d}x_0) \le h^{\theta} C_{t_0, d, \theta, T}.
\end{align*}
For $T_2$, using the fact that $\big| e^{-ia} - e^{-ib} \big| \le |e^{-ia}| |1 - e^{-i(b-a)}| \le |b-a|^{\theta}$, for any $\theta \in [0,1]$, we have:
\begin{align*}
T_2 
&\le \int_{\bR^d} \bigg| \frac{ \Big(\sum_{j = 1}^{n} t_{j}\xi_{j} \Big) \cdot (x - x_0) }{t} - \frac{ \Big( \sum_{j = 1}^{n} t_{j} \xi_{j} \Big) \cdot (x - x_0) }{t+h}\bigg|^{\theta} G_{t}(x-x_0) u_0(\mathrm{d}x_0)\\
&= \int_{\bR^d} \bigg| \frac{h \Big(\sum_{j = 1}^{n} t_{j}\xi_{j} \Big) \cdot (x - x_0)}{t(t+h)}
\bigg|^{\theta} G_{t}(x-x_0) u_0(\mathrm{d}x_0) \le h^{\theta} C_{t_0, \theta} \Big|\sum_{j = 1}^{n} t_{j}\xi_{j} \Big|^{\theta}.
\end{align*}
Therefore, there exist a constant $C^{(1)}$ depends on $t_0$, $T$, $\theta$, $d$ such that 
\begin{align*}
K_1 
&\le h^{\theta}  C^{(1)} \, \prod_{k = 1}^{n-1} \exp \Big( -\frac{1}{2} \frac{t_{k+1} - t_{k}}{t_{k} t_{k+1}} \Big| \sum_{j = 1}^{k} t_{j} \xi_{j} \Big|^2 \Big) \\
&\qquad \qquad \quad \times \exp \Big( -\frac{1}{2} \frac{t+h - t_{n}}{t_{n}(t+h)} \Big| \sum_{j = 1}^{n} t_{j} \xi_{j} \Big|^2 \Big) \times \Big( 1 + \Big|\sum_{j = 1}^{n} t_{j}\xi_{j} \Big|^{\theta} \Big)\\
&\le h^{\theta}  C^{(1)} \, \prod_{k = 1}^{n-1} \exp \Big( -\frac{1}{2} \frac{t_{k+1} - t_{k}}{t_{k} t_{k+1}} \Big| \sum_{j = 1}^{k} t_{j} \xi_{j} \Big|^2 \Big) \\
&\qquad \qquad  \quad \times \exp \Big( -\frac{1}{2} \frac{t - t_{n}}{t_{n} t} \Big| \sum_{j = 1}^{n} t_{j} \xi_{j} \Big|^2 \Big) \times \Big( 1 + \Big|\sum_{j = 1}^{n} t_{j}\xi_{j} \Big|^{\theta} \Big).
\end{align*}
We need the following inequality (see relation 5.2 in \cite{BC2018}): for any $A > 0$ and $\theta > 0$, there exists a constant $C_{\theta} > 0$ depending on $\theta$ such that 
\begin{align*}
%\label{relation4-20230419-6:19pm}
e^{-Ax^2} x^{2\theta} \le C_{\theta}  A^{-\theta} e^{-\frac{A}{2}x^2}, \text{ for all } x \ge 0.
\end{align*}
Using above inequality with $A = \frac{t - t_n}{t_n t}$ and $x = \Big|\sum_{j = 1}^{n} t_{j} \xi_{j} \Big|$, we obtain:
\begin{align*}
\exp \Big( -\frac{t - t_{n}}{t_{n} t} \Big| \sum_{j = 1}^{n} t_{j} \xi_{j} \Big|^2 \Big) \Big|\sum_{j = 1}^{n} t_{j}\xi_{j} \Big|^{2\theta} 
\le C_{\theta} \Big( \frac{t - t_n}{t_n t} \Big)^{- \theta} \exp \Big( - \frac{1}{2} \frac{t - t_{n}}{t_{n} t} \Big| \sum_{j = 1}^{n} t_{j} \xi_{j} \Big|^2 \Big).
\stepcounter{equation}\tag{\theequation}\label{relation1-20230420-11:39pm} 
\end{align*}
Note that for any $t \in [t_0, T]$, we have
\begin{align*}
\Big( \frac{t - t_n}{t_n t } \Big)^{- \theta} \le T^{2\theta}(t - t_n)^{-\theta}.
\stepcounter{equation}\tag{\theequation}\label{relation2-20230420-11:39pm} 
\end{align*}
Hence, denoting $t_{n+1} = t$, there exist a constant $C^{(2)}$ depends on $t_0$, $T$, $\theta$, $d$ such that 
\begin{align*}
K^2_1 
&\le h^{2\theta}  \big( C^{(1)} \big)^2
\bigg\{ \prod_{k = 1}^{n} \exp \Big( - \frac{1}{2} \frac{t_{k+1} - t_{k}}{t_{k} t_{k+1}} \Big| \sum_{j = 1}^{k} t_{j} \xi_{j} \Big|^2 \Big)\\ 
&\qquad + \prod_{k = 1}^{n-1} \exp \Big( -\frac{1}{2} \frac{t_{k+1} - t_{k}}{t_{k} t_{k+1}} \Big| \sum_{j = 1}^{k} t_{j} \xi_{j} \Big|^2 \Big)
\times C_{\theta} \Big( \frac{t - t_n}{t_n t} \Big)^{- \theta} \exp \Big( - \frac{1}{2} \frac{t - t_{n}}{t_{n} t} \Big| \sum_{j = 1}^{n} t_{j} \xi_{j} \Big|^2 \Big) \bigg\}\\
&\le h^{2\theta}  C^{(2)} \prod_{k = 1}^{n} \exp \Big( - \frac{1}{2} \frac{t_{k+1} - t_{k}}{t_{k} t_{k+1}} \Big| \sum_{j = 1}^{k} t_{j} \xi_{j} \Big|^2 \Big) \big[1 + (t - t_n)^{-\theta} \big].
\stepcounter{equation}\tag{\theequation}\label{rel1-20230908} 
\end{align*}

% K_2

We consider next the term $K_2$. Note that using the fact that $(1 - e^{-x})^2 \le 1 - e^{-x} \le x^{\theta}$ for all $x > 0$ and $\theta \in [0,1]$, we have:
\begin{align*}
\Big| I^{(1)}_1 - I^{(1)}_2 \Big|^2 
&= \exp \Big(  - \frac{t  - t_{n}}{t_{n} t} \Big| \sum_{j = 1}^{n} t_{j} \xi_{j} \Big|^2  \Big)  \bigg|1 - \exp \Big(  -\frac{1}{2} \frac{h}{t(t+h)} \Big| \sum_{j = 1}^{n} t_{j} \xi_{j} \Big|^2 \Big) \bigg) \bigg|^2\\
&\le \exp \Big( - \frac{t  - t_{n}}{t_{n} t} \Big| \sum_{j = 1}^{n} t_{j} \xi_{j} \Big|^2 \Big)
\Big( \frac{1}{2} \frac{h}{t(t+h)} \Big| \sum_{j = 1}^{n} t_{j} \xi_{j} \Big|^2 \Big)^{\theta} \\
&\le h^{\theta} C_{t_0, \theta, T} \Big( \frac{t - t_n}{t_n t } \Big)^{- \theta}  \exp \Big( - \frac{1}{2} \frac{t  - t_{n}}{t_{n} t} \Big| \sum_{j = 1}^{n} t_{j} \xi_{j} \Big|^2 \Big),
\end{align*}
where for the last inequality, we used the fact that $t(t+h) \ge t^2_0$, followed by relation \eqref{relation1-20230420-11:39pm}. Moreover, we see that $|I_{2}^{(2)}| \le w(t,x) \le C$, using that fact that $w$ is a continuous function on $[t_0, T] \times K$.
Therefore, we obtain:
\begin{align*}
K^2_2 
&\le h^{\theta} C^{(3)} \, (t - t_n)^{-\theta} \prod_{k = 1}^{n} \exp \Big( - \frac{1}{2} \frac{t_{k+1} - t_{k}}{t_{k} t_{k+1}} \Big| \sum_{j = 1}^{k} t_{j} \xi_{j} \Big|^2 \Big),
\stepcounter{equation}\tag{\theequation}\label{rel2-20230908} 
\end{align*}
where $C^{(3)} > 0$ is a constant depends on $t_0$, $T$, $\theta$, $d$.

We now give an upper bound for $\psi_{t,h,n}^{\alpha}(\pmb{t_n})$. By relations \eqref{rel1-20230908}, \eqref{rel2-20230908} and Lemma \ref{lem1-20230324-4:38pm}, we get:
\begin{align*}
\psi_{t,h,n}^{\alpha}(\pmb{t_n})
&\le \int_{(\bR^d)^n}  2(K^2_1 + K^2_2) \prod_{k = 1}^{n} |\xi_k|^{-\alpha} \mathrm{d}\boldsymbol{\xi} \le h^{\theta} C^{(4)} (t - t_n)^{-\theta} \prod_{j = 1}^{n} k_{\alpha}\Big( \frac{2(\frac{t_{j+1}}{2} - \frac{t_j}{2})(\frac{t_j}{2})}{\frac{t_{j+1}}{2}} \Big).
\end{align*}
To evaluate the last expression, we use \eqref{relation1-20230328-3:19pm}. We obtain that:
\begin{align*}
\psi_{t,h,n}^{\alpha}(\pmb{t_n}) \le h^{\theta} C^{(5)} \, t^{(d - \alpha)/2} \big[  t_1(t_2 - t_1) \ldots (t - t_n)  \big]^{-(d-\alpha)/2} (t - t_n)^{-\theta} ,
\end{align*}
where $C^{(5)}$ is a constant that depends on $d$, $t_0$, $T$, $a$, $b$, and $\beta$. 

Returning to \eqref{rel3-20230908-1:57pm}, we obtain:
\begin{align*}
A^{(\alpha)}_n(t,h) &\le h^{\theta} \Gamma_{0,t}^n \, n!  \, C^{(5)}  t^{(d - \alpha)/2} \\
&\qquad \quad \int_{0 < t_n < t} \int_{T_{n-1}(t_n)} \big[  t_1(t_2 - t_1) \ldots (t_n - t_{n-1})  \big]^{-\frac{d - \alpha}{2}} d \pmb{t_{n-1}}  (t-t_n)^{-\frac{d - \alpha}{2} - \theta} d t_n\\
& = h^{\theta} \Gamma_{0,t}^n \, n!  \, C^{(5)}  t^{(d - \alpha)/2} \frac{\Gamma\big( 1 - \frac{d - \alpha}{2} \big)^{n}}{\Gamma\big( n(1 - \frac{d - \alpha}{2} )\big)}
\int_{0}^{t}  t_n^{(n-1)(1 - \frac{d - \alpha}{2}) - \frac{d - \alpha}{2}}(t-t_n)^{-\frac{d - \alpha}{2} - \theta} d t_n\\
&\le h^{\theta} \Gamma_{0,t}^n \, n!  \, C^{(5)} \frac{\Gamma\big( 1 - \frac{d - \alpha}{2} \big)^{n}}{\Gamma\big( n(1 - \frac{d - \alpha}{2} )\big)} t^{\frac{d - \alpha}{2}} t^{(n-1)(1 - \frac{d - \alpha}{2} )}
\int_{0}^{t}  t_n^{- \frac{d - \alpha}{2}}(t-t_n)^{-\frac{d - \alpha}{2} - \theta} d t_n\\
&= h^{\theta} \Gamma_{0,t}^n \, n!  \, C^{(5)}
\frac{\Gamma\big( 1 - \frac{d - \alpha}{2} \big)^{n}}{\Gamma\big( n(1 - \frac{d - \alpha}{2} )\big)}  t^{n\big(1 - \frac{d - \alpha}{2} \big) - \theta} B\Big( 1 - \frac{d - \alpha}{2} , 1 -\frac{d - \alpha}{2} - \theta \Big),
\stepcounter{equation}\tag{\theequation}\label{rel4-20230909-9:13pm} 
\end{align*}
where we used Lemma \ref{lem1-20230328-3:39pm} in the first equality and $B(a,b)$ denotes the beta function. Note that $B( 1 - \frac{d - \alpha}{2}, 1 - \frac{d - \alpha}{2} - \theta )$ is finite since $1 - \frac{d - \alpha}{2} - \theta  > 0$ for any $\beta \in (\frac{d -a}{2} , 1)$ arbitrary. This expression can be bounded uniformly in $\alpha \in [a,b]$, using monotonicity properties of the $\Gamma$ function. We have:
\begin{align*}
\sup_{\alpha \in [a,b]} \sum_{n \ge 1} (p-1)^{n/2} \Big( \frac{1}{n!} A^{(\alpha)}_n(t,h) \Big)^{1/2}
\le C h^{(1-\beta)/2},
\end{align*}
where $C$ depends on ($d$, $t_0$, $T$, $p$, $a$, $b$, $\beta$). 
 
% B_n^{\alpha}(t,h)

Next, we examine $B_n^{\alpha}(t,h)$. Let $D_{t,h}=[0,t+h]^n \verb2\2 [0,t]^n$. Then
\begin{align*}
B_n^{\alpha}(t,h) & \leq \Gamma_{0,t+h}^n \int_{[0,t+h]^n} \gamma_{t,h,n}^{\alpha}(\pmb{t_n})1_{D_{t,h}}(\pmb{t_n}) d\pmb{t_n},
\stepcounter{equation}\tag{\theequation}\label{rel1-20230909-3:28pm} 
\end{align*}
where
%the function $\gamma_{t,h,n}^{\alpha}(\pmb{t_n})$
%is defined as follows: if $\pmb{t}_n \in [0,t]^n$ and $\rho$ is %the permutation of $1,\ldots,n$ such that %$t_{\rho(1)}<\ldots<t_{\rho(n)}$, we let %$u_j=t_{\rho(j+1)}-t_{\rho(j)}$ (with $t_{\rho(n+1)}=t$), and
\[
\gamma_{t,h,n}^{\alpha}(\pmb{t_n}):= (n!)^2 \int_{(\bR^d)^n} 
\big|\cF \widetilde{f}_{t+h,x,n}(\pmb{t_n}, \cdot)(\pmb{\xi_n}) \big|^2  \prod_{j=1}^{n}|\xi_j|^{-\alpha}d\pmb{\xi_n}.
\]
Note that for any $\bold{t} \in D_{t,h}$, there exists a permutation $\rho$ of $\{1, \ldots, n \}$ such that $ t_{\rho(1)}<\ldots<t_{\rho(n)}< t+h$ and $t < t_{\rho(n)} < t+h$. For the estimate below, we assume for simplicity that $\rho$ is the identity, i.e. $0 < t_1 < \ldots < t_n$ and $t < t_n < t+h$. Then, by Lemma \ref{lem1-20230317-4:21pm} and using the fact that $w(t,x) \le C$ for all $t \in [t_0, T]$ and $x \in K$, we get:
\begin{align*}
&\big|\cF f_{t+h,x,n}(\pmb{t_n}, \cdot)(\pmb{\xi_n}) \big|^2 \le C \prod_{k = 1}^{n-1} \exp \Big( - \frac{t_{k+1} - t_{k}}{t_{k} t_{k+1}} \Big| \sum_{j = 1}^{k} t_{j} \xi_{j} \Big|^2  \Big) \exp \Big( - \frac{t+h - t_{n}}{t_{n}(t+h)} \Big| \sum_{j = 1}^{n} t_{j} \xi_{j} \Big|^2 \Big).
\end{align*}
A similar relation holds for arbitrary $\rho$, with $t_j$, $\xi_j$ replaced by $t_{\rho(j)}$, $\xi_{\rho(j)}$ respectively. 

Using Lemma \ref{lem1-20230324-4:38pm}, we have:
\begin{align*}
&\int_{(\bR^d)^n} 
\big|\cF f_{t+h,x,n}(\pmb{t_n}, \cdot)(\pmb{\xi_n}) \big|^2  \prod_{j=1}^{n}|\xi_j|^{-\alpha}d\pmb{\xi_n} \\
& \qquad \le C
\int_{(\bR^d)^n} \prod_{k = 1}^{n-1} \exp \Big( - \frac{t_{k+1} - t_{k}}{t_{k} t_{k+1}} \Big| \sum_{j = 1}^{k} t_{j} \xi_{j} \Big|^2  \Big) \exp \Big( - \frac{t+h - t_{n}}{t_{n}(t+h)} \Big| \sum_{j = 1}^{n} t_{j} \xi_{j} \Big|^2 \Big)  \prod_{j=1}^{n}|\xi_j|^{-\alpha}d\pmb{\xi_n}\\
& \qquad \le  C \prod_{j = 1}^{n-1} k_{\alpha}\bigg( \frac{2(t_{j+1} - t_j)t_j}{t_{j+1}} \bigg) \times k_{\alpha}\bigg(  \frac{2(t+h - t_{n})t_{n}}{t+h}  \bigg)\\
& \qquad \le C_1 \big[  t_1(t_2 - t_1) \ldots (t_{n} - t_{n-1})  \big]^{-(d-\alpha)/2} \Big( \frac{t+h - t_n}{t+h} \Big)^{-(d-\alpha)/2},
\end{align*}
where we used \eqref{relation1-20230328-3:19pm} in the last inequality and $C_1 > 0$ is a constant that depends on $d$, $t_0$, $a$ and $b$.

We come back to \eqref{rel1-20230909-3:28pm}. Using the change of variable $s = t+h - t_n$, we have:
\begin{align*}
B_n^{\alpha}(t,h) &\le
\Gamma^n_{0,t+h} C_1 \, n!  \int_{t}^{t+h} \Big( \frac{t+h - t_n}{t+h} \Big)^{-(d-\alpha)/2} d t_{n}\\
&\qquad \qquad \Big( \int_{0< t_{1} < \ldots < t_{n}} 
\big[  t_1(t_2 - t_1) \ldots (t_{n} - t_{n-1})  \big]^{-(d-\alpha)/2} d t_1 \ldots  d t_{n-1} \Big)\\ 
&\le \Gamma^n_{0,t+h} C_1 \, n! \frac{\Gamma \Big( 1 - \frac{d-\alpha}{2} \Big)^{n}}{\Gamma \Big( n\big( 1 - \frac{d-\alpha}{2} \big) \Big)}
\int_{t}^{t+h} t_n^{(n-1)\big( 1 - \frac{d-\alpha}{2} \big) - \frac{d-\alpha}{2} } \Big( \frac{t+h - t_n}{t+h} \Big)^{-(d-\alpha)/2} d t_{n}\\
&= \Gamma^n_{0,t+h} C_1 \, n! \frac{\Gamma \Big( 1 - \frac{d-\alpha}{2} \Big)^{n}}{\Gamma \Big( n\big( 1 - \frac{d-\alpha}{2} \big) \Big)} \Big(\frac{1}{t+h} \Big)^{-(d-\alpha)/2}  \int_{0}^{h} (t+h - s)^{n\big( 1 - \frac{d-\alpha}{2} \big) - 1 } s^{-(d-\alpha)/2} d s\\
&\le \Gamma^n_{0,t+h} C_1 \, n! \frac{\Gamma \Big( 1 - \frac{d-\alpha}{2} \Big)^{n}}{\Gamma \Big( n\big( 1 - \frac{d-\alpha}{2} \big) \Big)} (t+h)^{(d-\alpha)/2} T^{n\big( 1 - \frac{d-\alpha}{2} \big) - 1 } \int_{0}^{h} s^{-(d-\alpha)/2} d s\\
&\le \Gamma^n_{0,t+h} C_2 \, n! \frac{\Gamma \Big( 1 - \frac{d-\alpha}{2} \Big)^{n}}{\Gamma \Big( n\big( 1 - \frac{d-\alpha}{2} \big) \Big)}  T^{(n-1)\big( 1 - \frac{d-b}{2} \big)} h^{1 - \beta},
\end{align*}
where in the last inequality we used the fact that $h^{1 -(d-\alpha)/2} \le h^{1 - \beta}$ since $\beta > \frac{d - a}{2} > \frac{d - \alpha}{2}$. Note that the preceding expression can be uniformly bounded for all $\alpha \in [a, b]$. Hence, we obtain:
\begin{align*}
\sup_{\alpha \in [a,b]} \sum_{n \ge 1} (p-1)^{n/2} \Big( \frac{1}{n!} B^{(\alpha)}_n(t,h) \Big)^{1/2} \le C \, h^{(1-\beta)/2},
\end{align*}
where $C$ is a constant depending on ($d$, $t_0$, $T$, $p$, $a$, $b$).

% space increments

{\bf Step 2 (space increments).} We denote $z = x' - x$. By hypercontractivity,
\[
\|u^{\alpha}(t,x+z)-u^{\alpha}(t,x)\|_p \leq \sum_{n\geq 1}(p-1)^{n/2} \|I_{n}^{\alpha}(f_{t,x+z,n}-f_{t,x,n})\|_p
\leq \sum_{n\geq 1}(p-1)^{n/2} \left( \frac{1}{n!} C_{n}^{\alpha}(t,z)\right)^{1/2},
\]
where $C_{n}^{\alpha}(t,z)=(n!)^2 \| \widetilde{f}_{t,x+z,n}-\widetilde{f}_{t,x,n}
\|_{\cH_{\alpha}^{\otimes n}}^{2}$. Note that
\begin{align*}
C_{n}^{\alpha}(t,z) & \leq \Gamma_{0,t}^n \int_{[0,t]^n}\psi_{t,z,n}^{\alpha}(\pmb{t_n}) d\pmb{t_n},
\stepcounter{equation}\tag{\theequation}\label{rel3-20230909-9:07pm} 
\end{align*}
where we fixed $\pmb{t_{n}} \in [0,t]^n$ and $\rho$ is the permutation of $1,\ldots,n$ such that $t_{\rho(1)}<\ldots<t_{\rho(n)}$ with $t_{\rho(n+1)}=t$ and
\[
\psi_{t,z,n}^{\alpha}(\pmb{t_n}) := 
\int_{(\bR^d)^n} \Big|\cF \Big(f_{t,x+z,n}(t_{\rho(1)}, \cdot, \ldots, t_{\rho(1)}, \cdot) - f_{t,x,n}(t_{\rho(1)}, \cdot, \ldots, t_{\rho(1)}, \cdot) \Big)(\pmb{\xi_n}) \Big|^2 
\prod_{j=1}^{n}|\xi_j|^{-\alpha}d\pmb{\xi_n}.
\]

Assume for simplicity that $t_1 < \ldots < t_n$. By Lemma \ref{lem1-20230317-4:21pm}, we have:
\begin{align*}
\cF \big(f_{t,x+z,n}(\pmb{t_n}, \cdot) - f_{t,x,n}(\pmb{t_n}, \cdot)\big)(\pmb{\xi_n}) = \prod_{k = 1}^{n} \exp \Big( -\frac{1}{2} \frac{t_{k+1} - t_{k}}{t_{k} t_{k+1}} \Big| \sum_{j = 1}^{k} t_{j} \xi_{j} \Big|^2 \Big) (J_1 + J_2)
\end{align*}
where
\begin{align*}
J_1 &:= \bigg( \exp \Big( -\frac{i}{t} \Big( \sum_{j = 1}^{n} t_{j} \xi_{j}\Big) \cdot (x+z) \Big) - \exp \Big( -\frac{i}{t} \Big( \sum_{j = 1}^{n} t_{j} \xi_{j}\Big) \cdot x \Big) \bigg)\\
& \qquad \quad \int_{\bR^d}  \exp\Big( -i \Big[ \sum_{j = 1}^{n} \Big( 1 - \frac{t_{j}}{t} \Big) \xi_{j} \Big] \cdot x_0  \Big) \, G_{t}(x + z-x_0) u_0(\mathrm{d}x_0)\\
J_2 &:=  \exp \Big( -\frac{i}{t} \Big( \sum_{j = 1}^{n} t_{j} \xi_{j}\Big) \cdot x \Big)\\
& \qquad \quad  \int_{\bR^d}  \exp\Big( -i \Big[ \sum_{j = 1}^{n} \Big( 1 - \frac{t_{j}}{t} \Big) \xi_{j} \Big] \cdot x_0  \Big) \, \Big( G_{t}(x + z - x_0) - G_{t}(x -x_0) \Big) u_0(\mathrm{d}x_0).
\end{align*}
%J_1
For the term $J_1$, using the fact that $|1 - e^{-ix}|^2 \le x^{2\theta}$, for any $\theta \in [0,1]$, we get
\begin{align*}
J^2_1 &\le \bigg| \exp \Big( -\frac{i}{t} \Big( \sum_{j = 1}^{n} t_{j} \xi_{j}\Big) \cdot z \Big) - 1 \bigg|^2 w_{+}^2(t,x+z)
\le C_1 \Big| \sum_{j = 1}^{n} t_{j} \xi_{j} \Big|^{2\theta} |z|^{2\theta},
\end{align*}
where $C_1 > 0$ is a constant depends on $t_0$, $d$, $\theta$ and $T$. Note that $|w_{+}(t,x+z)| = |w_{+}(t,x')| \le C$ for all $x' \in K$. 

We now turn to the analysis of $J_2$. We use the following inequality (see Lemma 4.1 of \cite{CH2019}): for any $\alpha \in (0,1]$, 
$$\big|G_t(x) - G_t(y) \big| \le \frac{C_{\alpha}}{t^{\alpha/2}} \big[  G_{2t}(x) + G_{2t}(y)  \big] |x - y|^{\alpha},$$
for any $t > 0$ and $x , y \in \bR^d$, where $C_{\alpha} > 0$ is a constant depending on $\alpha$. Using the above inequality, we have
\begin{align*}
J^2_2 &\le \Big|  \int_{\bR^d}  \exp\Big( -i \Big[ \sum_{j = 1}^{n} \Big( 1 - \frac{t_{j}}{t} \Big) \xi_{j} \Big] \cdot x_0  \Big) \, \Big( G_{t}(x + z - x_0) - G_{t}(x -x_0) \Big) u_0(\mathrm{d}x_0) \Big|^2\\
&\le \Big(  \int_{\bR^d}  \frac{C_{\theta}}{t^{\theta/2}} \big[ G_{2t}(x + z - x_0) + G_{2t}(x -x_0) \big]  |z|^{\theta} u_0(\mathrm{d}x_0)  \Big)^2 \le C_2 |z|^{2 \theta}.
\end{align*}
Here we used the fact that $w(2t, \cdot)$ is uniformly bounded on compact sets and $t \ge t_0$.

Hence, there exists a constant $C$ that depends on $t_0$, $d$, $\theta$ and $T$ such that 
\begin{align*}
&\Big|\cF \big(f_{t,x+z,n}(\pmb{t_n}, \cdot) - f_{t,x,n}(\pmb{t_n}, \cdot)\big)(\pmb{\xi_n}) \Big|^2  \le |z|^{2\theta} C ( F_1 + F_2 ),
\stepcounter{equation}\tag{\theequation}\label{rel1-20230909-9pm} 
\end{align*}
where
$$
F_1 := \prod_{k = 1}^{n} \exp \Big( - \frac{t_{k+1} - t_{k}}{t_{k} t_{k+1}} \Big| \sum_{j = 1}^{k} t_{j} \xi_{j} \Big|^2 \Big)  \Big| \sum_{j = 1}^{n} t_{j} \xi_{j} \Big|^{2\theta}
$$
and
$$
F_2 := \prod_{k = 1}^{n} \exp \Big( - \frac{t_{k+1} - t_{k}}{t_{k} t_{k+1}} \Big| \sum_{j = 1}^{k} t_{j} \xi_{j} \Big|^2 \Big).
$$

We only need to study $F_1$ since the term involving $F_2$ is uniformly bounded by Lemma \ref{lem1-20230324-4:38pm}. Using \eqref{relation1-20230420-11:39pm} and \eqref{relation2-20230420-11:39pm}, we have:
\begin{align*}
F_1 \le \prod_{k = 1}^{n} \exp \Big( -\frac{1}{2} \frac{t_{k+1} - t_{k}}{t_{k} t_{k+1}} \Big| \sum_{j = 1}^{k} t_{j} \xi_{j} \Big|^2 \Big) C_{\theta} T^{\theta} (t - t_n)^{- \theta}.
\stepcounter{equation}\tag{\theequation}\label{rel2-20230909-9pm} 
\end{align*}

By \eqref{rel1-20230909-9pm}, \eqref{rel2-20230909-9pm} and Lemma \ref{lem1-20230324-4:38pm}, we get:
\begin{align*}
\psi_{t,z,n}^{\alpha}(\pmb{t_n}) &\le |z|^{2\theta} C_3 \bigg( (t - t_n)^{- \theta} \prod_{j = 1}^{n} k_{\alpha}\Big( \frac{2(\frac{t_{j+1}}{2} - \frac{t_j}{2})(\frac{t_j}{2})}{\frac{t_{j+1}}{2}} \Big)
+ \prod_{j = 1}^{n} k_{\alpha}\Big( \frac{2(t_{j+1} - t_j)t_j}{t_{j+1}} \Big) 
\bigg)\\
&\le |z|^{2\theta} C_4 \bigg( \big[  t_1(t_2 - t_1) \ldots (t - t_n)  \big]^{-(d-\alpha)/2}(t - t_n)^{- \theta}  + \big[  t_1(t_2 - t_1) \ldots (t - t_n)  \big]^{-(d-\alpha)/2} \bigg)
\end{align*}
where we used relation \eqref{relation1-20230328-3:19pm} in the last inequality and $C_3$, $C_4$ are constants depending on $t_0$, $d$, $a$, $b$, $\theta$ and $T$. 

We come back to \eqref{rel3-20230909-9:07pm}. Similarly to \eqref{rel4-20230909-9:13pm}, using the estimate of $\psi_{t,z,n}^{\alpha}(\pmb{t_n})$ above, we have:
\begin{align*}
C^{(\alpha)}_n(t,z) &\le |z|^{2 \theta} \Gamma^n_{0,t} n! C_5 \bigg(
 \frac{ \Gamma \big( 1 - \frac{d-\alpha}{2} \big)^{n}}{\Gamma \big( n\big( 1 - \frac{d-\alpha}{2} \big) \big)} (T \vee 1)^{n(1 - \frac{d-\alpha}{2}) - \theta}
 + \frac{ \Gamma \big( 1 - \frac{d-\alpha}{2} \big)^{n+1}}{\Gamma \big( (n+1)\big( 1 - \frac{d-\alpha}{2} \big) \big)} (T \vee 1)^{n(1 - \frac{d-\alpha}{2})} \bigg)
\end{align*}
where requires again that $\beta \in (\frac{d - a}{2}, 1)$. Therefore,
\begin{align*}
\sup_{\alpha \in [a,b]} \sum_{n \ge 1} (p-1)^{n/2} \Big( \frac{1}{n!} C^{(\alpha)}_n(t,z) \Big)^{1/2} \le C |z|^{1-\beta},
\end{align*}
where $C$ is a constant depending on $d$, $t_0$, $T$, $p$, $a$, $b$ and $\beta$.
\end{proof}

% The proof of the main result

Now we are ready to prove Theorem \ref{main-th1}. Note that this proof employs similar concepts as those found in the proof of Theorem 2.1 of \cite{BL2023}.

\noindent{\bf Proof of Theorem \ref{main-th1}}:
%step 1
{\em Step 1. (finite-dimensional distribution convergence)} In this step, we prove that: for any $(t_1,x_1),\ldots,
(t_k,x_k)\in [t_0,T] \times \bR^d$,
\[
\big(u^{\alpha_n}(t_1,x_1),\ldots,u^{\alpha_n}(t_k,x_k)\big)
\stackrel{d}{\to}
\big(u^{\alpha^*}(t_1,x_1),\ldots,u^{\alpha^*}(t_k,x_k)\big).
\]
It is enough to prove that $u^{\alpha_n}(t,x) \to u^{\alpha^*}(t,x)$ in $L^2(\Omega)$ as $n\to \infty$, for any $(t,x) \in [t_0,T] \times \bR^d$. To do this, we approximate $u^{\alpha}(t,x)$ by the partial sum:
\begin{equation}
\label{eqn1-20230605-2:38pm}
u_m^{\alpha}(t,x)=w(t,x)+\sum_{k=1}^{m}I_k^{\alpha}(f_{t,x,k}).
\end{equation}
Note that for any $m\geq 1$, by Lemma \ref{Ik-conv}, 
\begin{align*}
\bE|u_m^{\alpha_n}(t,x)-u_m^{\alpha^*}(t,x)|^2 \leq m\sum_{k=1}^{m}
\bE|I_{k}^{\alpha_n}(f_{t,x,k})-I_{k}^{\alpha^*}(f_{t,x,k})|^2 \quad \mbox{as $n\to \infty$}.
\end{align*}
Moreover, $\bE|u_m^{\alpha^{\ast}}(t,x)-u^{\alpha^{\ast}}(t,x)|^2 \to 0, \text{ as } m \to \infty$. 

It remains to prove that
$$\sup_{n \ge 1} \bE|u_m^{\alpha_n}(t,x)-u^{\alpha_n}(t,x)|^2 \to 0, \text{ as } m \to \infty.$$
To see this, we choose specific values for $a$ and $b$ such that $\max(d-2,0)<a<\alpha^*<b<d$. Since $\alpha_n \to \alpha^*$, there exists $N \in \bN$ such that $a<\alpha_n<b$ for any $n\geq N$. Therefore, it suffices to prove that:
\begin{equation}
\label{uniform}
\sup_{\alpha \in [a,b]}\bE|u_m^{\alpha}(t,x)-u^{\alpha}(t,x)|^2
=\sup_{\alpha \in [a,b]} \sum_{k\geq m+1}\bE|I_k^{\alpha}(f_{t,x,k})|^2
\to 0 \quad \mbox{as $m \to \infty$.}
\end{equation}

Note that
\[
\bE|I_k^{\alpha}(f_{t,x,k})|^2 =k! \|\widetilde{f}_{t,x,k} \|_{\cH_{\alpha}^{\otimes k}}^{2} \leq  k! \Gamma_{0,t}^k
\int_{[0,t]^n}A_k^{\alpha}(\pmb{t_k})d\pmb{t_k},
\]
where $\Gamma_{0,t}=\int_{-t}^t \gamma_0(s)ds$ and
$
A_{k}^{\alpha}(\pmb{t_k})=\int_{(\bR^d)^k}
\big|\cF \widetilde{f}_{t,x,k}(\pmb{t_k},\bullet)(\pmb{\xi_k})\big|^2 \prod_{j=1}^k |\xi_j|^{-\alpha}d\pmb{\xi_k}$.
We denote $\pmb{t_k}=(t_1,\ldots,t_k) \in [0,t]^k$ be arbitrary and  $\rho$ be a permutation of $1,\ldots,k$ such that $t_{\rho(1)}<\ldots<t_{\rho(k)}$, with $t_{\rho(k+1)}=t$. Then by Lemma \ref{lem1-20230324-4:38pm}, we have:
\begin{align*}
&A_{k}^{\alpha}(\pmb{t_k})\\
&=\frac{1}{(k!)^2}
\int_{(\bR^d)^k} \bigg| \prod_{j = 1}^{k} \exp \bigg\{ -\frac{1}{2} \frac{t_{\rho(j+1)} - t_{\rho(j)}}{t_{\rho(j)} t_{\rho((j+1)}} \Big| \sum_{i = 1}^{j} t_{\rho(i)} \xi_{\rho(i)} \Big|^2 \bigg\}
\exp \bigg\{ -\frac{i}{t} \Big( \sum_{j = 1}^{k} t_{\rho(j)} \xi_{\rho(j)}  \Big) \cdot x \bigg\}\\
&\qquad \quad \times \int_{\bR^d} \exp\Big\{ -i \Big[ \sum_{j = 1}^{k} \Big( 1 - \frac{t_{\rho(j)}}{t} \Big) \xi_{\rho(j)} \Big] \cdot x_0  \Big\} G_t(x-x_0) u_0(\mathrm{d}x_0) \bigg|^2 \prod_{j = 1}^{k} |\xi_j|^{-\alpha}d\pmb{\xi_k}\\ 
&\leq \frac{w_{+}^2(t,x) }{(k!)^2} \prod_{j=1}^{k}k_{\alpha} \Big( \frac{2(t_{\rho(j+1)} - t_{\rho(j)}) t_{\rho(j)}}{t_{\rho(j+1)}} \Big)\\ 
&= \frac{w_{+}^2(t,x) (C^{(1)}_{d,\alpha})^k 2^{-(d-\alpha)k/2}}{(k!)^2} t^{\frac{d - \alpha}{2}} 
[t_{\rho(1)} (t_{\rho(2)} - t_{\rho(1)}) \ldots (t - t_{\rho(k)})]^{-\frac{d - \alpha}{2}},
\end{align*}
where $k_{\alpha}$ and $C^{(1)}_{d,\alpha}$ are given in \eqref{k_alpha_constant}, respectively \eqref{relation1-20230328-3:19pm}. Using Lemma \ref{lem1-20230328-3:39pm}, we get:
\begin{align*}
\bE|I_k^{\alpha}(f_{t,x,k})|^2 & \leq C k! \, \Gamma_{0,t}^k (C^{(1)}_{d,\alpha})^k 2^{-(d-\alpha)k/2} 
t^{\frac{d - \alpha}{2}} \\
&\qquad \qquad \qquad  \frac{1}{(k!)^2}\sum_{\rho \in S_n}\int_{t_{\rho(1)} < \ldots < t_{\rho(k)}} [t_{\rho(1)} (t_{\rho(2)} - t_{\rho(1)}) \ldots (t - t_{\rho(k)})]^{-\frac{d - \alpha}{2}} d\pmb{t_k}\\
&= C \, \Gamma_{0,t}^k (C^{(1)}_{d,\alpha})^k 2^{-(d-\alpha)k/2} t^{k(1 - \frac{d-\alpha}{2}) } 
\frac{ \big( \Gamma(1 - \frac{d-\alpha}{2}) \big)^{k+1}}{\Gamma \Big(   (k+1)  \big(1 - \frac{d-\alpha}{2} \big)\Big)}.
\end{align*}
The last thing we need to do is to uniformly bound above expression for all $\alpha \in [a,b]$. Note that for any $\alpha \in [a,b]$,
$$C^{(1)}_{d,\alpha} \le c_d \Big( \frac{1}{d-b} + \frac{2}{ 2 - (d - a)} \Big) =: C_{d,a,b}$$
and
$$2^{-(d-\alpha)k/2} \le 2^{-(d-b)k/2}.$$
%The remaining terms can be handled in a similar manner as demonstrated in the proof of Theorem 2.1 in \cite{BL2023}, using properties of gamma function. 
Moreover, $\Gamma(1 - \frac{d-\alpha}{2}) \le \Gamma(1 - \frac{d-a}{2})$, $t^{k(1 - \frac{d-\alpha}{2}) } \le (t \vee 1)^{k(1 - \frac{d-b}{2}) }$ and 
$$ \Gamma \Big(   (k+1)  \big(1 - \frac{d-\alpha}{2} \big)\Big) \geq \Gamma \Big(   (k+1)  \big(1 - \frac{d-a}{2} \big)\Big) \quad \mbox{for any} \ k \geq m_0,$$ 
where $m_0$ is chosen such that $m_0(1-\frac{d-a}{2})>x_0$, and therefore for any $m\geq m_0$,
\begin{align*}
&\sup_{\alpha \in [a,b]} \sum_{k\ge m+1} \mathbb{E}\big| I^{\alpha}_k(f_{t,x,k}) \big|^2  \le \sum_{k\ge m+1} C \Gamma^k_{0,T} C^k_{d,a,b} \frac{ \big( \Gamma(1 - \frac{d-a}{2}) \big)^{k+1}}{\Gamma \Big(   (k+1)  \big(1 - \frac{d-a}{2} \big)\Big)} (t \vee 1)^{k(1 - \frac{d-b}{2}) } \to 0.
\stepcounter{equation}\tag{\theequation}\label{sum-m-heat} 
\end{align*}
Relation \eqref{uniform} follows.

{\em Step 2. (tightness)} The fact that the sequence $(u^{\alpha_n})_{n\geq 1}$ is tight in $C([0,T] \times \bR^d)$ follows by Proposition 2.3 of \cite{yor}, using Theorem \ref{Th-uniform-C} above.

\qed

%%%%%%%%%%%%%%%%%%%%%%%%%%%%%%%%%%%%%%%%%%%%%%%%

%%%%%%%%%%%%%%%%%%%%%%%%%%%%%%%%%%%%%%%%%%%%%%%%
\section{Rough noise}
\label{Rough-noise}

In this section, we consider equation \eqref{PAM} driven by a Gaussian noise $W$ where the temporal covariance function $\gamma_0$ is given by \eqref{temporal-cov}, and the measure $\mu$ is given by \eqref{def-mu2}. The existence of the solution was proved in \cite{BCY2022}.

To emphasize the dependence on the parameter $H$, we denote the noise, the Hilbert space, and the solution, by $W^{H},\cH^{H},u^{H}$, respectively. Denoting the multiple integral of order $n$ with respect to $W^{H}$ as $I_n^{H}$, the series expansion \eqref{series} takes the following form:
\[
u^{H}(t,x)=1+\sum_{n\geq 1}I_{n}^{H}(f_{t,x,n}).
\]

Similarly to Section 4 of \cite{BL2023}, all Gaussian processes $\{ W^{H} \}_{H \in (0,1/2)}$ can be constructed on the same probability space. Let $\widehat{W}$ be the $\bC$-valued Gaussian measure given by \eqref{def-W-hat}. 
We recall the following representation of the multiple integral with respect to the process $W^{H}$: for any $t \mapsto \varphi(t, \cdot) \in \mathcal{S}'(\bR^d)$,
\begin{equation}
\label{In-H}
I_k^{H}(\varphi)=c_{H_0}^{k/2} c_H^{k/2}\int_{(\bR \times \bR)^k} \cF_t[\cF_x \varphi(\pmb{t_k},\bullet)(\pmb{\xi_k})](\pmb{\tau_k})
\prod_{j=1}^{k}|\tau_j|^{\frac{1}{2}-H_0}|\xi_j|^{\frac{1}{2}-H}
\widehat{W}(d\tau_1,d\xi_1) \ldots \widehat{W}(d\tau_k,d\xi_k).
\end{equation}

We will use the following lemmas.

\begin{lemma}[Lemma 3.1 of \cite{ref5}]
\label{lemma2-20220530-4:52}
Let $G$ be the fundamental solution of the heat equation, the integral $\int_{\bR} |\mathcal{F} G_t(\xi)|^2 |\xi|^{\alpha} \mathrm{d}\xi$ is finite if and only if $\alpha > -1$ and in this case,
$$\int_{\bR} |\mathcal{F} G_t(\xi) |^2 |\xi|^{\alpha} \mathrm{d}\xi = \Gamma \Big( \frac{1+\alpha}{2} \Big) \ 
t^{-(1+\alpha)/2}.$$
\end{lemma}

\begin{lemma}[Lemma A.1 in \cite{BCY2022}]
\label{lem1-20230823-11:07am}
Suppose that $\alpha_i > -1$, $\beta_i > -1$ for any $i = 1, \ldots, n$, and
\begin{equation}
\label{condition-20230831-5:14pm}
\sum_{i = 1}^{k}(\alpha_i + \beta_i) + k + 1 + \alpha_{k+1} > 0 \, \text{ for all } k = 1, \ldots, n-1.
\end{equation}
Then by setting $t_{n+1} = t$, $|\alpha| = \sum_{i = 1}^{n} \alpha_i$ and $|\beta| = \sum_{i = 1}^{n} \beta_i$, we have that 
\begin{align*}
&I_n(t, \alpha_1, \ldots, \alpha_n, \beta_1, \ldots, \beta_n) := \int_{0 < t_1 < \ldots < t_n < t} \prod_{i = 1}^{n} t_i^{\alpha_i} (t_{i+1} - t_i)^{\beta_i} \mathrm{d}\bold{t}\\
&= \frac{\Gamma(\alpha_1+1) \prod_{i = 1}^{n}\Gamma(\beta_i+1)}{\Gamma \big(|\alpha| + |\beta| + n + 1 \big)} \prod_{k = 1}^{n-1} 
\frac{\Gamma\big( \sum_{i = 1}^{k}(\alpha_i + \beta_i) + k + 1 + \alpha_{k+1} \big)}{\Gamma\big( \sum_{i = 1}^{k}(\alpha_i + \beta_i) + k + 1 \big)} t^{|\alpha| + |\beta| + n}. 
\end{align*}
\end{lemma}

We will use a similar approximation technique as in Section \ref{Regular-noise}. More precisely, we study first the $L^2(\Omega)$-continuity in $H$ of the multiple integral of $I_k^H(f_{t,x,k})$ for fixed $k$. This is achieved by the following lemma. Note that the proof is the same as in the proof of Lemma 4.2 of \cite{BL2023}, the difference being that here we use a different expression for the Fourier transform $\cF f_{t,x,k}(\pmb{t_k},\cdot)(\pmb{\xi_k})$, due to the general initial condition. The following inequality is also required, which is a consequence of the Littlewood-Hardy inequality: for any $\varphi \in L^{1/H_0}(\bR_{+}^k)$,
\begin{equation}
\label{LH-ineq}
\alpha_{H_0}^k \int_{\bR_{+}^k} \int_{\bR_{+}^k}
\prod_{j=1}^{k}|t_j-s_j|^{2H_0-2} \varphi(\pmb{t_k})
\varphi(\pmb{s_k}) d\pmb{t_k}d\pmb{s_k} \leq b_{H_0}^k\left(
\int_{\bR_{+}^k}|\varphi(\pmb{t_k})|^{1/H_0} d\pmb{t_k} \right)^{2H_0},
\end{equation}
where $b_{H_0}>0$ depends only on $H_0$.

\begin{lemma}
\label{rough-conv-Ik}
Let $\ell= \max(3/4-H_0,0)$.
If $H_n \to H^* \in (\ell,1/2)$, then
\[
Q_n:=\bE|I_k^{H_n}(f_{t,x,k})-I_k^{H^*}(f_{t,x,k})|^2 \to 0, \quad \mbox{as $n\to \infty$},
\]
for any $t>0$, $x \in \bR$ and $k\geq 1$.
\end{lemma}

\begin{proof}
Proceeding as in the proof of Lemma 4.2 of \cite{BL2023}, by \eqref{LH-ineq}, we infer that

\begin{align*}
Q_n \leq k! \, b_{H_0}^k \left(\int_{T_k(t)} A_k^{(n)}(\pmb{t_k},\pmb{t_k})^{\frac{1}{2H_0}}d\pmb{t_k}\right)^{2H_0}.
\end{align*}
where
\begin{align*}
A_k^{(n)}(\pmb{t_k},\pmb{s_k})=\int_{\bR^k}
 \phi_{\pmb{\xi_k}}(\pmb{t_k})  \overline{\phi_{\pmb{\xi_k}}(\pmb{s_k})}  
 \left|  c^k_{H_n}\prod_{j=1}^k|\xi_j|^{\frac{1}{2} - H_n} - c^k_{H^{\ast}}\prod_{j=1}^k|\xi_j|^{\frac{1}{2} - H^{\ast}}\right|^2 d\pmb{\xi_k}
\end{align*}
and $\phi_{\pmb{\xi_k}}(\pmb{t_k})$ is defined by \eqref{def-phi-k} with $d = 1$. 

Therefore, it suffices to show that:
\begin{equation}
\label{conv-A-H0}
\int_{T_k(t)} A_k^{(n)}(\pmb{t_k},\pmb{t_k})^{\frac{1}{2H_0}}d\pmb{t_k}\to 0, \quad \mbox{as $n\to \infty$}.
\end{equation}
For this, we apply the Dominated Convergence Theorem. We will show that
\begin{equation}
\label{A-conv}
A_k^{(n)}(\pmb{t_k},\pmb{t_k})\to 0 \quad \mbox{as $n\to \infty$},
\end{equation}
and there exists a function $h_{k}(\pmb{t_k})$ such that
\begin{equation}
\label{domin}
A_k^{(n)}(\pmb{t_k},\pmb{t_k}) \leq h_{k}(\pmb{t_k}) \quad \mbox{and} \int_{T_k(t)}h_{k}(\pmb{t_k}) d\pmb{t_k}<\infty.
\end{equation}

We first prove \eqref{A-conv}. By the change of variables $z_i = t_i \xi_i$ for $i = 1, \ldots, j$, followed by $\eta_j = \sum_{i = 1}^{j} z_i$ for $j = 1, \ldots, k$ (with $\eta_0 = 0$), we have
\begin{align*}
&A_k^{(n)}(\pmb{t_k},\pmb{t_k})
\le  w_{+}^2(t,x) \int_{\bR^k} \prod_{j = 1}^{k} \exp \Big( - \frac{t_{j+1} - t_{j}}{t_{j} t_{j+1}} \Big| \sum_{i = 1}^{j} t_{i} \xi_{i} \Big|^2 \Big)\\
&\qquad \qquad \qquad  \qquad \qquad \qquad 
\bigg| c^{k/2}_{H_n} \prod_{j=1}^{k} \big| \xi_j \big|^{\frac{1}{2} - H_n} - c^{k/2}_{H^{\ast}} \prod_{j=1}^{k} \big| \xi_j \big|^{\frac{1}{2} - H^{\ast}} \bigg|^2 d\boldsymbol{\xi_k}\\
&=  w_{+}^2(t,x) \Big( \prod_{j = 1}^{k} t_j^{-1} \Big) \int_{\bR^k} \prod_{j = 1}^{k} \exp \Big( - \frac{t_{j+1} - t_{j}}{t_{j} t_{j+1}} \Big| \sum_{i = 1}^{j} z_i \Big|^2 \Big)\\
&\qquad \qquad \qquad  
\bigg| c^{k/2}_{H_n} \prod_{j = 1}^{k} t_j^{H_n - \frac{1}{2}} \prod_{j=1}^{k} \big| z_j \big|^{\frac{1}{2} - H_n} - c^{k/2}_{H^{\ast}} \prod_{j = 1}^{k}t_j^{H^{\ast} - \frac{1}{2}} \prod_{j=1}^{k} \big| z_j \big|^{\frac{1}{2} - H^{\ast}} \bigg|^2 d\boldsymbol{z_k}\\
&=  w_{+}^2(t,x) \int_{\bR^k} \prod_{j = 1}^{k} \exp \Big( - \frac{t_{j+1} - t_{j}}{t_{j} t_{j+1}} \Big| \eta_j \Big|^2 \Big)\\
&\qquad  
\bigg| c^{k/2}_{H_n} \prod_{j = 1}^{k}t_j^{H_n - 1} \prod_{j=1}^{k} \big| \eta_j - \eta_{j-1} \big|^{\frac{1}{2} - H_n} - c^{k/2}_{H^{\ast}} \prod_{j = 1}^{k}t_j^{H^{\ast} - 1} \prod_{j=1}^{k} \big| \eta_j - \eta_{j-1} \big|^{\frac{1}{2} - H^{\ast}} \bigg|^2 d \boldsymbol{\eta_k}.
\stepcounter{equation}\tag{\theequation}\label{relation2-20230903-2:51pm} 
\end{align*}
We denote by $B^{(n)}(\pmb{t_k},\pmb{\eta_k})$ the integrand in the last integral above. Clearly, $B^{(n)}(\pmb{t_k},\pmb{\eta_k}) \to 0$ as $n\to \infty$. So relation \eqref{A-conv} will follow by the Dominated Convergence Theorem, as long as we show that there exists a function $B(\pmb{t_k},\pmb{\eta_k})$ such that
\[
B^{(n)}(\pmb{t_k},\pmb{\eta_k})\leq B(\pmb{t_k},\pmb{\eta_k}) \quad \mbox{and} \quad \int_{\bR^k}B(\pmb{t_k},\pmb{\eta_k})d
\pmb{\eta_k}<\infty.
\]
Note that $B^{(n)}(\pmb{t_k},\pmb{\eta_k})\leq 2\big(B_1^{(n)}(\pmb{t_k},\pmb{\eta_k})+B_2(\pmb{t_k},
\pmb{\eta_k})\big)$, where
\[
B_1^{(n)}(\pmb{t_k},\pmb{\eta_k}) := \prod_{j = 1}^{k} \exp \Big( - \frac{t_{j+1} - t_{j}}{t_{j} t_{j+1}} \big| \eta_j \big|^2 \Big) c^{k}_{H_n} \prod_{j = 1}^{k} t_j^{2H_n - 2} \prod_{j=1}^{k} \big| \eta_j - \eta_{j-1} \big|^{1 - 2H_n}
\]
and $B_2$ has a similar expression with $H_n$ replaced by $H^*$. It suffices to consider $B_1^{(n)}$.

Fix numbers $a$ and $b$ such that
\begin{equation}
\label{def-ab}
\ell<a<H^*<b<\frac{1}{2}.
\end{equation}
Then $H_n \in [a,b]$ for $n$ large enough.
Since the constant $c_H$ defines a continuous function of $H$, $c_{H_n}$ is bounded by a constant $c$.
We use the inequality
\begin{equation}
\label{prod-ineq}
\prod_{j=1}^{k}|\eta_j-\eta_{j-1}|^{1-2H} \leq \sum_{\pmb{a} \in A_k} \prod_{j=1}^{k}|\eta_j|^{(1-2H)a_j},
\end{equation}
where $A_k$ is a set of multi-indices $\pmb{a}=(a_1,\ldots,a_k)$ such that
$a_1 \in \{1,2\}$, $a_n \in \{0,1\}$, $a_j \in \{0,1,2\}$ for $j=1,\ldots,k-1$, $\sum_{j=1}^{k}a_j=k$, and ${\rm card}(A_k)=2^{k-1}$. Note that $|\eta_j|^{(1-2H)a_j} \leq f_{a}(|\eta_j|)$ for any $j=1,\ldots,n$, where $a$ is the constant from $\{0, 1, 2 \}$ and the functions $f_0,f_1,f_2$ are defined as follows: $f_0(r)=1$ for any $r> 0$,
\[
f_1(r)=
\left\{
\begin{array}{ll}
r^{1-2a}  & \mbox{if $r\geq 1$} \\
1 & \mbox{if $r\in (0,1)$}
\end{array} \right.
\quad
\mbox{and}
\quad
f_2(r)=
\left\{
\begin{array}{ll}
r^{2(1-2a)}  & \mbox{if $r\geq 1$} \\
1 & \mbox{if $r\in (0,1)$}.
\end{array} \right.
\]
Moreover, for any $H_n \in [a,b]$,
\begin{align*}
\prod_{j=1}^{k} t^{2H_n - 2}_j &= \Big( \prod_{j=1}^{k} t^{-1}_j \Big)^{2-2H_n} \le \prod_{j=1}^{k} \big( t^{-1}_j  \vee 1 \big)^{2-2a} \le \prod_{j=1}^{k} \big( t_j  \wedge 1 \big)^{2a-2} \le \prod_{j=1}^{k} \big( t_j  \vee 1 \big)^{2a-2}.
\end{align*}
It follows that:
\begin{align*}
B_1^{(n)}(\pmb{t_k},\pmb{\eta_k}) &\le c^k \Big(  \prod_{j = 1}^{k} (t_j \vee 1)^{2a-2}  \Big) \sum_{\bold{a} \in A_k} \prod_{j = 1}^{k} \exp \Big( - \frac{t_{j+1} - t_{j}}{t_{j} t_{j+1}} \Big| \eta_j \Big|^2 \Big)  \prod_{j = 1}^{k}  f_{a_j}(|\eta_j|) =:F(\pmb{t_k},\pmb{\eta_k}).
\end{align*}
By the same argument, $B_2^{(n)}(\pmb{t_k},\pmb{\eta_k})\leq F(\pmb{t_k},\pmb{\eta_k})$. Hence, $B^{(n)}(\pmb{t_k},\pmb{\eta_k})\leq 2F(\pmb{t_k},\pmb{\eta_k})$ and so,
\begin{equation}
\label{bound-A}
A_{k}^{(n)}(\pmb{t_k},\pmb{t_k})=\int_{\bR^k} B^{(n)}(\pmb{t_k},\pmb{\eta_k})d\pmb{\eta_k} \leq
4 \int_{\bR^k} F(\pmb{t_k},\pmb{\eta_k})d\pmb{\eta_k}.
\end{equation}
It remains to prove that $\int_{\bR^k} F(\pmb{t_k},\pmb{\eta_k})d\pmb{\eta_k}<\infty$. We give below an estimate for this integral which will be used for the proof of relation \eqref{domin}. Note that
\[
\int_{\bR^k} F(\pmb{t_k},\pmb{\eta_k})d\pmb{\eta_k}=c^k \prod_{j = 1}^{k} (t_j \vee 1)^{2a-2} 
\sum_{\bold{a} \in A_k} \prod_{j = 1}^{k} I(a_j),  \ \mbox{with} \
I(a_j)=\int_{\bR}e^{-\frac{t_{j+1} - t_{j}}{t_{j} t_{j+1}} |\eta_j|^2}f_{a_j}(|\eta_j|)d\eta_j.
\]
We use the same arguments for the study of the upper bound of $I(a_j)$ in the proof of Lemma 4.2 in \cite{BL2023} with $t_{j+1} - t_{j}$ replaced by $\frac{t_{j+1} - t_{j}}{t_{j} t_{j+1}}$. Moreover, we have
\[
\begin{cases}
(t_j \vee 1)^{2a - 2} (t_j t_{j+1})^{\frac{1}{2}} \le t^{\frac{1}{2}}_{j+1} \frac{1}{(t_j \vee 1)^{2 - 2a}} (t_j \vee 1)^{\frac{1}{2}} \le t^{\frac{1}{2}}_{j+1}, &\text{ since } \frac{1}{2} < 2 - 2a,\\[0.8em]
(t_j \vee 1)^{2a - 2} (t_j t_{j+1})^{1-a} \le t^{1-a}_{j+1} \frac{1}{(t_j \vee 1)^{1 - a}} (t_j \vee 1)^{1-a} \le t^{1-a}_{j+1}, &\text{ since } 1-a < 2 - 2a,\\[0.8em]
(t_j \vee 1)^{2a - 2} (t_j t_{j+1})^{\frac{3-4a}{2}} \le t^{\frac{3-4a}{2}}_{j+1} \frac{1}{(t_j \vee 1)^{\frac{3-4a}{2}}} (t_j \vee 1)^{\frac{3-4a}{2}} \le t^{\frac{3-4a}{2}}_{j+1}, &\text{ since } \frac{3-4a}{2} < 2 - 2a.
\end{cases}
\]
Since $\frac{1}{2} < 1 - a < \frac{3 - 4a}{2}$, $t^{\frac{1}{2}}_{j+1}$, $t^{1-a}_{j+1}$ and $t^{\frac{3-4a}{2}}_{j+1}$ are all dominated by $ (T \vee 1)^{\frac{3-4a}{2}}$. Hence,
\begin{align*}
(t_j \vee 1)^{2a - 2} I(a_j) &\le c_a t^{\frac{3-4a}{2}}_{j+1} \Big\{ (t_{j+1} - t_j)^{-\frac{1}{2}} + (t_{j+1} - t_j)^{-(1-a)} + (t_{j+1} - t_j)^{-\frac{3 -4a}{2}}   \Big\}\\
&\le c_a c_{t,a} (T \vee 1)^{\frac{3-4a}{2}} \, (t_{j+1} - t_j)^{-\frac{3 -4a}{2}}.
\end{align*}
Since $\text{card}(A_k) = 2^{k-1}$, we obtain:
\begin{equation}
\label{relation3-20230903-8:30pm}
\int_{\bR^k}  F(\pmb{t_k},\pmb{\eta_k}) \mathrm{d} \boldsymbol{\eta}  \le 2^{k-1} c^k c^k_a c^k_{t,a} (T \vee 1)^{k \cdot \frac{3-4a}{2}} \prod_{j=1}^{k}  (t_{j+1} - t_j)^{-\frac{3 - 4a}{2}}.
\end{equation}
This proves the integrability of $F(\pmb{t_k},\pmb{\eta_k})$ and so, relation \eqref{A-conv} follows by the Dominated Convergence Theorem. 

\medskip

Now we prove \eqref{domin}. By relations \eqref{bound-A} and \eqref{relation3-20230903-8:30pm},  we have
\begin{align*}
\Big( A_k^{(n)}(\pmb{t_k},\pmb{t_k}) \Big)^{\frac{1}{2H_0}} 
&\le \Big(  2^{k+1} c^k c_a^k c^k_{t,a} (T \vee 1)^{k \cdot \frac{3-4a}{2}} \prod_{j=1}^{k} (t_{j+1} - t_j)^{-\frac{3 - 4a}{2}} \Big)^{\frac{1}{2H_0}} =: h_k(\bold{t}).
\end{align*}
Note that $h_k(\bold{t})$ is integrable over the symplex $T_k(t)$ since $-\frac{3 - 4a}{4H_0} > -1$ which is equivalent to $a > \frac{3}{4} - H_0$. This finishes the proof of \eqref{domin} and the justification of the application of the Dominated Convergence Theorem.

\end{proof}

We now prove some moment estimates for the increments of the solution $u^H$, uniform in $H$, which will be used below for the proof of tightness of the family $(u^H)_{H \in [a,b]}$. In particular, Theorem \ref{rough-unif-mom} shows that the process $u^H$ has a continuous modification. We will work with this modification.

\begin{theorem}
\label{rough-unif-mom}
Let $\ell= \max(3/4-H_0,0)$. Let $[a,b]$ be a compact set in $[0,1]$ such that:
\begin{equation}
\label{range-ab}
\ell<a<b<\frac{1}{2}.
\end{equation}
Let $u^H$ be the solution of the heat equation \eqref{PAM} with noise $W^H$. For any $p\geq 2$, $T>0$, $c_0 \in (0,\frac{2H_0+a-1}{2H_0})$ and
\begin{equation}
\label{cond-delta}
0<\theta<2H_0(1-c_0)+a-1,
\end{equation}
there exist a constant $C>0$ such that for any $t,t'\in [t_0,T]$ and $x,x'\in K$,
\[
\sup_{H \in [a,b]}\bE|u(t,x)-u(t',x')|^p \leq C \Big(|t-t'|^{p\theta/2}+|x-x'|^{p\theta}\Big).
\]
\end{theorem}

\begin{proof}
{\bf Step 1 (time increments).} Let $t \in [t_0,T]$ and $h>0$ be such that $t+h \in [t_0,T]$. As in the proof of Theorem \ref{Th-uniform-C},
\begin{align*}
\big\|u^{H}(t+h,x)-u^{H}(t,x)\big\|_p 
& \leq \sum_{n\geq 1}(p-1)^{n/2}
\left(\frac{2}{n!} \Big(A_n^{H}(t,h)+B_n^{H}(t,h)\Big) \right)^{1/2},
\end{align*}
where
$A_n^{H}(t,h)=(n!)^2\| \widetilde{f}_{t+h,x,n}1_{[0,t]^n}-\widetilde{f}_{t,x,n}
\|_{\cH_{H}^{\otimes n}}^2$ and $B_n^{H}(t,h)=(n!)^2\| \widetilde{f}_{t+h,x,n}1_{[0,t+h]^n - [0,t]^n}\|_{\cH_{H}^{\otimes n}}^2$.

% A_n^{\alpha}(t,h)
{\em We examine $A_n^{H}(t,h)$}. Note that by inequality \eqref{LH-ineq}, 
\begin{align*}
A^{H}_n(t,h) &\le b^n_{H_0} \bigg( \int_{[0,t]^n} \big( \psi_{t,h,n}^{H}(\pmb{t_n}) \big)^{\frac{1}{2H_0}} \mathrm{d}{\bold t}  \bigg)^{2H_0},
\stepcounter{equation}\tag{\theequation}\label{relation11-20230831-10:18pm} 
\end{align*}
where the function $\psi_{t,h,n}^{H}(\pmb{t_n})$ is defined as follows: if $\pmb{t_{n}} \in [0,t]^n$ is fixed and $\rho$ is the permutation of $1,\ldots,n$ such that $t_{\rho(1)}<\ldots<t_{\rho(n)}$ with $t_{\rho(n+1)}=t$, and
\begin{align*}
\psi_{t,h,n}^{H}(\pmb{t_n}) & :=
 \int_{\bR^n} \Big|\cF \Big(f_{t+h,x,n}(t_{\rho(1)}, \cdot, \ldots, t_{\rho(n)}, \cdot) - f_{t,x,n}(t_{\rho(1)}, \cdot, \ldots, t_{\rho(n)}, \cdot)\Big)(\pmb{\xi_n}) \Big|^2 
\prod_{j=1}^{n}|\xi_j|^{1-2H}d\pmb{\xi_n}.
\end{align*}
Assume for simplicity that $t_1 < \ldots < t_n$. To estimate $\psi_{t,h,n}^{H}(\pmb{t_n})$, we use estimates \eqref{rel1-20231205-11:51pm}, \eqref{rel1-20230908} and \eqref{rel2-20230908}, for the increments of the Fourier transform appearing above. Hence, there exists a constant $C^{(1)}$ depending on $t_0$, $\theta$ and $T$ such that 
\begin{align*}
&\psi_{t,h,n}^{H}(\pmb{t_n}) \le h^{\theta} C^{(1)}
\big(t - t_{n}\big)^{-\theta} I^{(n)}_{t/2}(\pmb{t_n}/2)
\stepcounter{equation}\tag{\theequation}\label{rel2-20231205-5:05pm}
\end{align*}
where
\begin{align*}
I^{(n)}_{t}(\pmb{t_n}) := c_H^n \int_{\bR^n} \prod_{k = 1}^{n} \exp\bigg\{ - \frac{t_{k+1} - t_{k}}{t_{k+1}t_{k}} \Big| \sum_{i=1}^{k} t_{i} \xi_{i} \Big|^2 \bigg\} 
|\xi_k|^{1-2H} d\pmb{\xi_n}.
\stepcounter{equation}\tag{\theequation}\label{rel2-20231008-2:25pm} 
\end{align*}
and $t_{n+1} = t$. 

We study $I^{(n)}_{t}(\pmb{t_n})$. Using the change of variables $z_i = t_i \xi_i$ for $i = 1, \ldots, k$, followed by $\eta_k = \sum_{i = 1}^{k} z_i$ for $k = 1, \ldots, n$, we have:
\begin{align*}
&I^{(n)}_{t}(\pmb{t_n}) \\
&\le c_H^n \Big( \prod_{k = 1}^{n} t_k \Big)^{2H - 2} \int_{\bR^n} 
\prod_{k = 1}^{n} \exp\bigg\{ - \frac{t_{k+1} - t_{k}}{t_{k+1}t_{k}}  |\eta_k |^2 \bigg\} |\eta_1|^{1-2H}  \prod_{k = 2}^{n} \big( |\eta_{k}|^{1-2H} + |\eta_{k-1}|^{1-2H} \big) \mathrm{d}\boldsymbol{\eta_n}\\
&\le c_H^n \Big( \prod_{k = 1}^{n} t_k \Big)^{2H - 2} 
\sum_{\pmb{\alpha_n} \in D_n^{(H)}} \prod_{k=1}^{n} \bigg\{ \int_{\bR}  \exp \Big(  - \frac{t_{k+1} - t_{k}}{t_{k+1}t_{k}} |\eta_k|^2\Big) |\eta_k|^{\alpha_k} \mathrm{d}\eta_k \bigg\},
%\stepcounter{equation}\tag{\theequation}\label{relation1-20230830-5:13pm} 
\end{align*}
where for the last inequality, we used relation \eqref{prod-ineq} and the set $D_n^{(H)}$ is the set of multi-indices $\pmb{\alpha_n}=(\alpha_1,\ldots,\alpha_n)$ with $\alpha_j=(1-2H)a_j$ and $\pmb{a_n}=(a_1,\ldots,a_n)\in A_n$. Using Lemma \ref{lemma2-20220530-4:52} to compute the $d\eta_k$ integrals above with $k=1,\ldots,n$, we have:
\begin{align*}
I^{(n)}_{t}(\pmb{t_n}) &\le c_H^n  \Big( \prod_{k = 1}^{n} t_k \Big)^{2H - 2} 
\sum_{\pmb{\alpha_n} \in D_n^{(H)}}  \prod_{k = 1}^{n} \Gamma
\Big(\frac{1+ \alpha_k}{2} \Big) \Big( \frac{t_{k+1} - t_k}{t_k t_{k+1}} \Big)^{- \frac{1 + \alpha_k}{2}}\\
&\quad \le C_{H,1}^n \sum_{\pmb{\alpha_n} \in D_n^{(H)}} t^{\frac{1 + \alpha_n}{2}} t_1^{\frac{4H - 3 + \alpha_1}{2}} 
\Big( \prod_{k=2}^{n} t_k^{\frac{4H - 2 + \alpha_{k-1} + \alpha_k }{2}} \Big) 
\prod_{k=1}^{n} (t_{k+1} - t_k)^{- \frac{1 + \alpha_k}{2}}
\stepcounter{equation}\tag{\theequation}\label{relation1-20230831-3:59pm}
\end{align*}
where for the second inequality we used the constant $C_{H,1}$ defined by 
\begin{equation}
\label{constantC_H_1}
C_{H,1} = c_H \max \bigg\{ \Gamma \Big( \frac{1}{2} \Big), \Gamma \Big( 1 - H \Big), \Gamma \Big( \frac{3 - 4H}{2} \Big)  \bigg\}.
\end{equation} 
Both $c_H$ and $C_{H,1}$ can be uniformly bounded for all $H \in [a,b]$. Hence, we obtain:
\begin{align*}
&\big(t - t_{n}\big)^{-\theta} I^{(n)}_{t/2}(\pmb{t_n}/2) \\
&\le (C^{(1)})^n 2^n \sum_{\pmb{\alpha_n} \in D_n^{(H)}}  t^{\frac{1 + \alpha_n}{2}} t_{1}^{\frac{4H - 3 + \alpha_1}{2}} 
\Big( \prod_{k=2}^{n} t_{k}^{\frac{4H - 2 + \alpha_{k-1} + \alpha_k }{2}} \Big)
\Big( \prod_{k=1}^{n-1} \big( t_{k+1} - t_{k} \big)^{- \frac{1 + \alpha_k}{2}} \Big) (t - t_{n})^{- \frac{1 + \alpha_n}{2} - \theta}.
\stepcounter{equation}\tag{\theequation}\label{rel1-20231104-2:55pm}
\end{align*}
Taking power $\frac{1}{2H_0}$ on both sides of \eqref{rel2-20231205-5:05pm} above, we obtain:
\begin{align*}
&\Big( \psi^{H}_{t,z,n}(\pmb{t_n}) \Big)^{\frac{1}{2H_0}} \le 
\Big( h^{\theta}  C^{(1)} 2^n \Big)^{\frac{1}{2H_0}} \sum_{\pmb{\alpha_n} \in D_n^{(H)}} 
\prod_{k = 1}^{n} t^{\widetilde{\alpha}_k}_{k}  \prod_{k = 1}^{n-1} \big( t_{k+1} - t_{k}\big)^{\widetilde{\beta}_k}
\big( t - t_{n}\big)^{\widetilde{\beta}_{n, \theta}}
\stepcounter{equation}\tag{\theequation}\label{rel2-20231104-3:43pm}
\end{align*}
where
\begin{align*}
\widetilde{\alpha}_k =
  \begin{cases}
    \frac{4H - 3 + \alpha_1}{4H_0},     & \quad k = 1\\[0.8em]
    \frac{4H - 2+ \alpha_{k - 1} + \alpha_{k}}{4H_0},   & \quad  k = 2, \ldots, n
  \end{cases}
\stepcounter{equation}\tag{\theequation}\label{rel1-20231212}
\end{align*}
and 
\begin{align*}
\widetilde{\beta}_k = - \frac{1 + \alpha_k}{4H_0}, \; k = 1, \ldots, n-1
\stepcounter{equation}\tag{\theequation}\label{rel2-20231212}
\end{align*}
and
\begin{align*}
\widetilde{\beta}_{n, \theta} = - \frac{1 + \alpha_n + 2\theta}{4H_0}.
\stepcounter{equation}\tag{\theequation}\label{rel3-20231212}
\end{align*}
A similar estimate holds for arbitrary $(t_1, \ldots, t_n) \in [0, t]^n$ with $t_{\rho(1)} < \ldots < t_{\rho(n)}$ for same permutation $\rho \in S_n$, where $S_n$ is the set of all permutation of $\{1, \ldots, n \}$. From here, we derive that
\begin{align*}
&\int_{[0,t]^n}\left(\psi_{t,h,n}^{H}(\pmb{t_n})\right)^{\frac{1}{2H_0}} d\pmb{t_n} \\
& \qquad \leq 
\Big( h^{\theta}  C^{(1)} 2^n \Big)^{\frac{1}{2H_0}} n! 
\sum_{\pmb{\alpha_n} \in D_n^{(H)}}
\int_{T_n(t)}
 \prod_{k = 1}^{n} t^{\widetilde{\alpha}_k} _{k} \prod_{k = 1}^{n-1} \big( t_{k+1} - t_{k}\big)^{\widetilde{\beta}_k}
\big( t_{n+1} - t_{n}\big)^{\widetilde{\beta}_{n, \theta}} d\pmb{t_n}.
\stepcounter{equation}\tag{\theequation}\label{psi-b}
\end{align*}
To compute the integral
\[
{\cal I}(\pmb{\alpha_n}):= \int_{T_n(t)}
 \prod_{k = 1}^{n} t^{\widetilde{\alpha}_k} _{k} \prod_{k = 1}^{n-1} \big( t_{k+1} - t_{k}\big)^{\widetilde{\beta}_k}
\big( t_{n+1} - t_{n}\big)^{\widetilde{\beta}_{n, \theta}} d\pmb{t_n},
\]
we use Lemma \ref{lem1-20230823-11:07am}. To apply this Lemma, we first need $\widetilde{\alpha}_k > -1$ for all $k = 1, \ldots, n$ and $\widetilde{\beta}_k > -1$ for all $k = 1, \ldots, n-1$ which hold since $\alpha_k > 0$ for all $k = 1, \ldots, n$. We further need to check $\widetilde{\beta}_{n,\theta} > -1$ and condition \eqref{condition-20230831-5:14pm} holds for $k = 1, \ldots, n-1$. When $\alpha_n = 0$, we need $-\frac{1 + 2\theta}{4H_0} > -1$, i.e. $\theta < 2H_0 - 1/2$. When $\alpha_n = 1 - 2H$, we need $-\frac{2 + 2H + 2\theta}{4H_0} > -1$, i.e. $\theta < 2H_0 + H - 1$. Note that $2H_0 - 1/2 > 2H_0 + H - 1$ since $H < 1/2$. Hence, for all $H \in [a,b]$ arbitrary, we encounter the condition
\begin{equation}
\label{cond1-20231020-5:02pm}
0 < \theta < 2H_0 + a - 1.
\end{equation}
The verification of \eqref{condition-20230831-5:14pm} has been done in \cite{YP2020}. Under condition \eqref{cond1-20231020-5:02pm}, we can apply Lemma \ref{lem1-20230823-11:07am} to deduce that:
\begin{align*}
{\cal I}(\pmb{\alpha_n}) = \frac{\Gamma(\widetilde{\alpha}_1 + 1) 
\prod_{k = 1}^{n-1}\Gamma( \widetilde{\beta}_k + 1) \Gamma( \widetilde{\beta}_{n,\theta} + 1)}{\Gamma \big(|\widetilde{\alpha}| + |\widetilde{\beta}| + n + 1 \big)} \gamma_n 
t^{|\widetilde{\alpha}| + |\widetilde{\beta}| + n}
\stepcounter{equation}\tag{\theequation}\label{rel1-20231210-3:20pm}
\end{align*}
with $$\gamma_n = \gamma_n(\pmb{\alpha_n}) = \prod_{k = 1}^{n-1} 
\frac{\Gamma\big( \sum_{i = 1}^{k}(\widetilde{\alpha}_i + \widetilde{\beta}_i) + k + 1 + \widetilde{\alpha}_{k+1} \big)}{\Gamma\big( \sum_{i = 1}^{k}(\widetilde{\alpha}_i + \widetilde{\beta}_i) + k + 1 \big)}.$$
Note that $\gamma_n$ does not depend on $\widetilde{\beta}_{n, \theta}$, and 
$$\gamma_n \le 1.$$ 
This proof require significant effort. We refer the reader to Step 4 of the proof of Theorem 1.1 in \cite{BCY2022}. 

To bound these remaining factors in ${\cal I}(\pmb{\alpha_n})$, we use some monotonicity properties of the Gamma function: $\Gamma$ is decreasing on $(0,x_0)$ and increasing on $(x_0,\infty)$, where $x_0 \approx 1.4$. Therefore, for any $\alpha_j \in \{0,1-2H,2(1-2H)\}$ and $H \in [a,b]$,
we have $\Gamma(\widetilde{\alpha}_1 + 1) \le \Gamma \big( \frac{2H_0 + a - 1}{2H_0} \big)$ and $\prod_{k = 1}^{n-1}\Gamma(\widetilde{\beta}_k + 1) \le \big( \Gamma\big( 1 - \frac{3-4a}{4H_0} \big) \big)^{n-1}$. Next, we observe that $\widetilde{\beta}_{n,\theta} + 1 = 1 -\frac{1+\alpha_n + 2\theta}{4H_0}$ takes values in the interval $[1 - \frac{1- a + \theta}{2H_0},1-\frac{1}{4H_0}]$ whose lower bound may be close to 0. Since $\lim_{x\to 0}\Gamma(x)=\infty$, we control this term by fixing an arbitrary value $c_0\in (0,\frac{2H_0 + a -1}{2H_0})$, and then
choosing
\begin{equation}
\label{cond-theta-rough}
0< \theta < 2H_0(1-c_0) + a - 1 
\end{equation}
With this choice of $\theta$, \eqref{cond1-20231020-5:02pm} holds, and
more importantly $1 - \frac{1- a + \theta}{2H_0}>c_0$, so that for any $\alpha_n\in \{0,1-2H\}$ and $H \in [a,b]$,
\begin{equation}
%\label{G-c0}
\Gamma\Big(1 -\frac{1+\alpha_n + 2\theta}{4H_0} \Big) \le \Gamma(c_0).
\end{equation}
For any $H \in [a,b]$, the upper bound of $|\widetilde{\alpha}| + |\widetilde{\beta}| + n$ is
\begin{align*}
|\widetilde{\alpha}| + |\widetilde{\beta}| + n &\le  \frac{n(2H_0 + H - 1)}{2H_0} - \frac{1 + \alpha_n}{4H_0} \le  \frac{n(2H_0 + H - 1)}{2H_0} - \frac{1}{4H_0}\\
&\le  \frac{n(2H_0 + b - 1)}{2H_0} - \frac{1}{4H_0} \le \frac{n(2H_0 + b - 1)}{2H_0} 
\end{align*}
and the lower bound of $|\widetilde{\alpha}| + |\widetilde{\beta}| + n$ is
\begin{align*}
|\widetilde{\alpha}| + |\widetilde{\beta}| + n 
&\ge \frac{(n-1)(2H_0 + H - 1)}{2H_0} - \frac{1 + \alpha_n}{4H_0} \ge \frac{(n-1)(2H_0 + H - 1)}{2H_0} -  \frac{1-H}{2H_0}\\
&\ge \frac{(n-1)(2H_0 + a - 1)}{2H_0} -  \frac{1-a}{2H_0}.
\end{align*}
To bound the $\Gamma$-value appearing in the denominator, we pick an integer $m_0$ such that $\frac{(m_0-1)(2H_0 + a - 1)}{2H_0} -  \frac{1-a}{2H_0} > x_0$. Then, for any $n\geq m_0$ and for any $H \in [a,b]$,
\begin{align*}
\Gamma \Big(  |\widetilde{\alpha}| + |\widetilde{\beta}| + n + 1  \Big)
>  \Gamma\Big( \frac{(n-1)(2H_0 + a - 1)}{2H_0} -  \frac{1-a}{2H_0} + 1 \Big) > c^{n-1} \big[ (n-1)! \big]^{\frac{2H_0 + a -1 }{2H_0}}.
%\stepcounter{equation}\tag{\theequation}\label{rel4-20230831-10:00pm} 
\end{align*}
Hence, combining all the estimates above, for any $\pmb{\alpha_n} \in D_{n}^{(H)}$ and for any $H \in [a,b]$,
\begin{align*}
{\cal I}(\pmb{\alpha_n}) \leq \frac{\big( C^{(2)} \big)^{n-1}}{[(n-1)!]^{\frac{2H_0+a-1}{2H_0}}}(t\vee 1)^{\frac{n(2H_0+b-1)}{2H_0}}.
\end{align*}
Returning to \eqref{psi-b}, we have:
\[
\int_{[0,t]^n}\Big(\psi_{t,h,n}^{H}(\pmb{t_n})\Big)^{\frac{1}{2H_0}}
d\pmb{t_n} \leq h^{\frac{\theta}{2H_0}} \frac{\big( C^{(3)} \big)^{n-1}}{[(n-1)!]^{\frac{a-1}{2H_0}}}.
\]

Finally, coming back to \eqref{relation11-20230831-10:18pm}, we obtain:
\[
A_n^H(t,h) \leq h^{\theta}  \frac{ \big( C^{(4)} \big)^{(n-1)2H_0}}{[(n-1)!]^{a-1}}.
\]
Consequently, for any $p\geq 2$, $t \in [t_0,T]$, $H \in [a,b]$ and $\theta$ as in \eqref{cond-theta-rough},
\begin{equation}
\label{bound-A-H}
\sum_{n\geq 1}(p-1)^{n/2}\left(\frac{1}{n!}A_n^H(t,h)\right)^{1/2} \leq C h^{\theta/2}.
\end{equation}
where $C$ depends on $p$, $a$, $b$, $\theta$, $c_0$, $H_0$, and $T$.

%B_n^{H}(t,h)

{\em We examine $B_n^{H}(t,h)$}. Let $D_{t,h}=[0,t+h]^n \verb2\2 [0,t]^n$. By inequality \eqref{LH-ineq},
\begin{equation}
\label{B-bd}
B_n^H(t,h) \leq b_{H_0}^n \left(\int_{[0,t+h]^n} \gamma_{t,h,n}^{H}(\pmb{t_n})^{\frac{1}{2H_0}}1_{D_{t,h}}
(\pmb{t_n})d\pmb{t_n},
\right)^{2H_0}
\end{equation}
where 
\[
\gamma_{t,h,n}^{H}(\pmb{t_n}) := (n!)^2 c_H^n \int_{\bR^n} 
\big|\cF \widetilde{f}_{t+h,x,n}(\pmb{t_n}, \cdot)(\pmb{\xi_n}) \big|^2  \prod_{j=1}^{n}|\xi_j|^{1-2H} d\pmb{\xi_n}.
\]
Note that for any $\bold{t} \in D_{t,h}$, there exists a permutation $\rho$ of $\{1, \ldots, n \}$ such that $ t_{\rho(1)}<\ldots<t_{\rho(n)}< t+h$ and $t < t_{\rho(n)} < t_{\rho(n+1)} = t+h$. For the estimate below, we assume for simplicity that $\rho$ is the identity, i.e. $0 < t_1 < \ldots < t_n$ and $t < t_n < t_{n+1} = t+h$. Then, by Lemma \ref{lem1-20230317-4:21pm} and using the fact that $w(t,x) \le C$ for all $t \in [t_0, T]$ and $x \in K$, we get:
\begin{align*}
&\big|\cF f_{t+h,x,n}(\pmb{t_n}, \cdot)(\pmb{\xi_n}) \big|^2 \le C \prod_{k = 1}^{n-1} \exp \Big( - \frac{t_{k+1} - t_{k}}{t_{k} t_{k+1}} \Big| \sum_{j = 1}^{k} t_{j} \xi_{j} \Big|^2  \Big) \exp \Big( - \frac{t+h - t_{n}}{t_{n}(t+h)} \Big| \sum_{j = 1}^{n} t_{j} \xi_{j} \Big|^2 \Big).
\end{align*}
A similar relation holds for arbitrary $\rho$, with $t_j$, $\xi_j$ replaced by $t_{\rho(j)}$, $\xi_{\rho(j)}$ respectively. Hence, using the same approach as in \eqref{rel2-20231008-2:25pm} and \eqref{relation1-20230831-3:59pm} and Lemma \ref{lemma2-20220530-4:52}, we have:
\begin{align*}
& \gamma_{t,h,n}^{H}(\pmb{t_n})
\le C \; c_H^n  \Big( \prod_{k=1}^{n} t_{k}\Big)^{2H-2}
\sum_{\pmb{\alpha_n} \in D_{n}^{(H)}}  \Big( \prod_{k = 1}^{n-1}\int_{\bR} 
\exp \Big( - \frac{t_{k+1} - t_{k}}{t_{k} t_{k+1}} \Big) |\eta_k|^{\alpha_k}  \mathrm{d}\eta_k \Big) \\
&\qquad \qquad \qquad \qquad \qquad  \qquad \qquad   \qquad \qquad \times \Big(\int_{\bR}  \exp \Big( - \frac{t+h - t_{n}}{t_{n}(t+h)} \Big) |\eta_n|^{\alpha_n}  \mathrm{d}\eta_n \Big) \\
& = C c_H^n  \Big( \prod_{k=1}^{n} t_{k} \Big)^{2H-2}
\sum_{\pmb{\alpha_n} \in D_{n}^{(H)}} 
\prod_{k = 1}^{n-1} \Gamma\Big( \frac{1+\alpha_k}{2}  \Big)
\Big( \frac{t_{k+1} - t_{k}}{t_{k+1}t_{k}} \Big)^{-\frac{1+\alpha_k}{2}} \Gamma\Big( \frac{1+\alpha_n}{2}  \Big) \Big( \frac{t+h - t_{n}}{t_{n}(t+h)} \Big)^{-\frac{1+\alpha_n}{2}}\\
&\le C C_{H,1} \sum_{\pmb{\alpha_n} \in D_{n}^{(H)}} t_{1}^{\frac{4H - 3 + \alpha_1}{2}} 
\Big( \prod_{k=2}^{n} t_{k}^{\frac{4H - 2 + \alpha_{k-1} + \alpha_k }{2}} \Big) (t+h)^{\frac{1 + \alpha_n}{2}} \prod_{k=1}^{n-1} 
\big( t_{k+1} - t_{k} \big)^{- \frac{1 + \alpha_k}{2}} \big( t+h - t_{n} \big)^{-\frac{1+\alpha_n}{2}}
\stepcounter{equation}\tag{\theequation}\label{relation1-20230901-2:58pm} 
\end{align*}
where $C_{H,1}$ is given by \eqref{constantC_H_1} and it can be uniformly bounded for all $H \in [a,b]$.

Hence, taking power $\frac{1}{2H_0}$ on both sides of relation \eqref{relation1-20230901-2:58pm}, we obtain:
\begin{align*}
\Big(  \gamma_{t,h,n}^{H}(\pmb{t_n}) \Big)^{\frac{1}{2H_0}} 
&\le 
\big( C^{(5)} \big)^{\frac{1}{2H_0}}  \sum_{\pmb{\alpha_n} \in D_{n}^{(H)}} \prod_{k = 1}^{n} t_{k}^{\widetilde{\alpha}_k} \big( t_{k+1} - t_{k}\big)^{\widetilde{\beta}_k} (t+h)^{\frac{1 + \alpha_n}{4H_0}},
\stepcounter{equation}\tag{\theequation}\label{relation2-20230901-4:00pm} 
\end{align*}
where $\widetilde{\alpha}_1 , \ldots, \widetilde{\alpha}_n$ are given by \eqref{rel1-20231212}; $\widetilde{\beta}_1, \ldots, \widetilde{\beta}_{n-1}$ are given by \eqref{rel2-20231212} and 
$$
\widetilde{\beta}_n = - \frac{1 + \alpha_n}{4H_0}
$$
Integrating over the set $D_{t,h}$ and using Lemma \ref{lem1-20230823-11:07am}, we obtain: 
\begin{align*}
&\int_{[0,t+h]^n} \Big(  \gamma_{t,h,n}^{H}(\pmb{t_n}) \Big)^{\frac{1}{2H_0}} 1_{D_{t,h}}({\bold t}) d \bold{t} \\
&\quad \le \big( C^{(5)} \big)^{\frac{1}{2H_0}}  n! 
\int_{t}^{t+h} \bigg( \int_{T_{n-1}(t_{n})} \sum_{\pmb{\alpha_n} \in D_{n}^{(H)}} (t+h)^{\frac{1 + \alpha_n}{4H_0}}  \prod_{k = 1}^{n-1} t^{\widetilde{\alpha}_k}_{k} (t_{k+1} - t_{k} )^{\widetilde{\beta}_k} (t+h - t_{n} )^{\widetilde{\beta}_n} d \pmb{t_{n-1}} \bigg) d t_{n}\\
&\quad \le \big( C^{(5)} \big)^{\frac{1}{2H_0}}  n!  \sum_{\pmb{\alpha_n} \in D_{n}^{(H)}} (t+h)^{\frac{1 + \alpha_n}{4H_0}}
\int_{t}^{t+h} J(t_n)  \cdot (t + h - t_{n})^{\widetilde{\beta}_n}\mathrm{d} t_{n}
\stepcounter{equation}\tag{\theequation}\label{rel2-20230902-4:20pm} 
\end{align*}
where
\begin{align*}
J(t_n) &:= \int_{T_{n-1}(t_{n})} 
\prod_{k=1}^{n-1} t_k^{\widetilde{\alpha}_k} (t_{k+1} - t_k)^{\widetilde{\beta}_k} d \pmb{t_{n-1}} \\
&= \frac{\Gamma(\widetilde{\alpha}_1 + 1) \prod_{k = 1}^{n-1}\Gamma(\widetilde{\beta}_k + 1)}{\Gamma \big(\sum_{k=1}^{n-1}\widetilde{\alpha}_k + \sum_{k=1}^{n-1}\widetilde{\beta}_k  + n  \big)} \gamma_{n-1} 
t^{\sum_{k=1}^{n-1}\widetilde{\alpha}_k + \sum_{k=1}^{n-1}\widetilde{\beta}_k + n - 1}
\stepcounter{equation}\tag{\theequation}\label{relation2-20230901-8:06pm} 
\end{align*}
with 
$$\gamma_{n-1} = \gamma_{n-1}(\pmb{\alpha_{n-1}}) = \prod_{k = 1}^{n-2} 
\frac{\Gamma\big( \sum_{i = 1}^{k}(\widetilde{\alpha}_i + \widetilde{\beta}_i) + k + 1 + \widetilde{\alpha}_{k+1} \big)}{\Gamma\big( \sum_{i = 1}^{k}(\widetilde{\alpha}_i + \widetilde{\beta}_i) + k + 1 \big)}.
$$
Note that $\gamma_{n-1} \le 1$, by using the same arguments in Step 4 of the proof of Theorem 1.1 in \cite{BCY2022}. To bound these factors in \eqref{relation2-20230901-8:06pm}, we use the same approach as in the study of ${\cal I}(\pmb{\alpha_n})$ above. It follows that:
\begin{align*}
J(t_n) \le \frac{ \big( C^{(6)} \big)^{n-1}}{\big[ (n-1)! \big]^{\frac{2H_0 + a - 1}{2H_0}}} \; (t \vee 1)^{\frac{(n-1)(2H_0 + b - 1)}{2H_0}}.
\end{align*}
Using the change of variable $s = t+h-t_n$, we conclude that
\begin{align*}
\int_{[0,t+h]^n}\gamma_{t,h,n}^{(H)}(\pmb{t_n})^{\frac{1}{2H_0}}
1_{D_{t,h}}(\pmb{t_n}) d\pmb{t_n} \leq  \frac{ \big( C^{(7)} \big)^{n-1}}{\big[ (n-1)! \big]^{\frac{2H_0 + a - 1}{2H_0}}} \sum_{\pmb{\alpha_n}\in D_n^{(H)}}
\frac{1}{1-\frac{1+\alpha_n}{4H_0}}h^{1-\frac{1+\alpha_n}{4H_0}}.
\end{align*}
Note that $(1-\frac{1+\alpha_n}{4H_0})^{-1}$ is bounded by $\frac{2H_0}{2H_0+a-1}$ for any $\alpha_n \in \{0,1-2H\}$ and $H \in [a,b]$. Moreover, due to condition \eqref{cond1-20231020-5:02pm},
\[
\frac{\theta}{2H_0}<\frac{2H_0+a-1}{2H_0}\leq 1-\frac{1+\alpha_n}{4H_0}\leq 1-\frac{1}{4H_0}.
\]
Therefore,
\[
\int_{[0,t+h]^n}\gamma_{t,h,n}^{(H)}(\pmb{t_n})^{\frac{1}{2H_0}}
1_{D_{t,h}}(\pmb{t_n}) d\pmb{t_n} \leq h^{\frac{\theta}{2H_0}} \frac{ \big( C^{(8)} \big)^{n-1} }{\big[ (n-1)! \big]^{\frac{2H_0 + a - 1}{2H_0}}} .
\]
Returning to \eqref{B-bd}, we infer that
$B_n(t,h) \leq  h^{\theta} \frac{C^{n-1}}{[(n-1)!]^{a-1}}$,
and hence, for any $p\geq 2$, $t \in [t_0,T]$ and $\theta$ satisfying \eqref{cond1-20231020-5:02pm} (and in particular, for any $\theta$ chosen as in \eqref{cond-theta-rough}),
\[
\sum_{n\geq 1}(p-1)^{n/2}\left(\frac{1}{n!}B_n(t,h)  \right)^{1/2} \leq C h^{\theta/2}.
\]
Here $C$ is a constant depending on $p$, $a$, $b$, $c_0$, $H_0$, and $T$.

% space increments

{\bf Step 2 (space increments).} We denote $z = x' - x$. As in the proof of Theorem \ref{Th-uniform-C},
\[
\|u^{H}(t,x+z)-u^{H}(t,x)\|_p \leq \sum_{n\geq 1}(p-1)^{n/2} \|I_{n}^{H}(f_{t,x+z,n}-f_{t,x,n})\|_p
\leq \sum_{n\geq 1}(p-1)^{n/2} \left( \frac{1}{n!} C_{n}^{H}(t,z)\right)^{1/2},
\]
where $C_{n}^{H}(t,z)=(n!)^2 \| \widetilde{f}_{t,x+z,n}-\widetilde{f}_{t,x,n}
\|_{\cH_{H}^{\otimes n}}^{2}$. By \eqref{LH-ineq}, we have:
\begin{align*}
C^{H}_n(t,z) 
&\le b^n_{H_0} \bigg( \int_{[0,t]^n} \Big( \Psi^{H}_{t,z,n}(\pmb{t_n}) \Big)^{\frac{1}{2H_0}} d\pmb{t_n}  \bigg)^{2H_0},
\stepcounter{equation}\tag{\theequation}\label{relation7-20230902-10:53pm} 
\end{align*}
where we fixed $\pmb{t_{n}} \in [0,t]^n$ and $\rho$ is the permutation of $1,\ldots,n$ such that $t_{\rho(1)}<\ldots<t_{\rho(n)}$ with $t_{\rho(n+1)}=t$ and
\[
\Psi^{H}_{t,z,n}(\pmb{t_n}) := 
\int_{\bR^n} \Big|\cF \Big(f_{t,x+z,n}(t_{\rho(1)}, \cdot, \ldots, t_{\rho(1)}, \cdot) - f_{t,x,n}(t_{\rho(1)}, \cdot, \ldots, t_{\rho(1)}, \cdot) \Big)(\pmb{\xi_n}) \Big|^2 
\prod_{j=1}^{n}|\xi_j|^{1 - 2H}d\pmb{\xi_n}.
\]
Assume for simplicity that $t_1 < \ldots < t_n$. To study $\Psi^{H}_{t,z,n}(\pmb{t_n})$, we can use the same estimates \eqref{rel1-20230909-9pm} and \eqref{rel2-20230909-9pm}, for the increments of the Fourier transform appearing above. Using the definition of $I^{(n)}_{t}(\pmb{t_n})$ in \eqref{rel2-20231008-2:25pm}, there exist a constant $C$ denepdning on $t_0$, $\theta$ and $T$ such that
\begin{align*}
\Psi^{H}_{t,z,n}(\pmb{t_n}) \le |z|^{2\theta} C \big( I_1 + I_2 \big)
\stepcounter{equation}\tag{\theequation}\label{rel1-20230902-8:08pm} 
\end{align*}
where
\begin{align*}
I_1 &:=  C_{\theta} T^{\theta} (t - t_{n})^{- \theta} c_H^n \int_{\bR^n} \prod_{k = 1}^{n} \exp\bigg\{ -\frac{1}{2} \frac{t_{k+1} - t_{k}}{t_{k+1}t_{k}} \Big| \sum_{j=1}^{k} t_{j} \xi_{j} \Big|^2 \bigg\} \big|\xi_{k}\big|^{1-2H} d\pmb{\xi_n}\\
&\le C_{\theta} T^{\theta} (t - t_{n})^{- \theta}  I^{(n)}_{t/2}(\pmb{t_n} / 2)
\end{align*}
and
\begin{align*}
I_2 &:= c_H^n \int_{\bR^n} \prod_{k = 1}^{n} \exp\bigg\{ - \frac{t_{k+1} - t_{k}}{t_{k+1}t_{k}} \Big| \sum_{j=1}^{k} t_{j} \xi_{j} \Big|^2 \bigg\} \big|\xi_{k}\big|^{1-2H} d\pmb{\xi_n}= I^{(n)}_{t}(\pmb{t_n}).
\end{align*}
Taking power $1/(2H_0)$ on \eqref{rel1-20230902-8:08pm}, we obtain:
\begin{align*}
&\Big( \Psi^{H}_{t,z,n}(\pmb{t_n}) \Big)^{\frac{1}{2H_0}}  \le |z|^{\frac{\theta}{H_0}} C^{\frac{1}{2H_0}} \bigg(   
C^{\frac{1}{2H_0}} _{\theta} T^{\frac{\theta}{2H_0}} \Big( (t - t_{n})^{-\theta}I^{(n)}_{t/2}(\pmb{t_n} / 2) \Big)^{\frac{1}{2H_0}} + \Big(I^{(n)}_{t}(\pmb{t_n}) \Big)^{\frac{1}{2H_0}} 
\bigg).
\end{align*}

We estimate separately the two terms in the inequality above. We start with the first summand. Using \eqref{rel1-20231104-2:55pm}, we have
\begin{align*}
&C^{\frac{1}{2H_0}} _{\theta} T^{\frac{\theta}{2H_0}} \Big( (t - t_{n})^{-\theta}I^{(n)}_{t/2}(\pmb{t_n} / 2) \Big)^{\frac{1}{2H_0}} \\
&\qquad \le \Big( C_{\theta} T^{\theta} (C^{(2)})^{n} 2^{n} \Big)^{\frac{1}{2H_0}}
\sum_{\boldsymbol{\alpha} \in D^{(H)}_n} t^{\frac{1+\alpha_n}{4H_0}} \prod_{k = 1}^{n} t^{\widetilde{\alpha}_k} _{k}  
\prod_{k = 1}^{n-1} \big( t_{k+1} - t_{k}\big)^{\widetilde{\beta}_k}
\big( t - t_{n}\big)^{\widetilde{\beta}_{n, \theta}}
\end{align*}
where where $\widetilde{\alpha}_1 , \ldots, \widetilde{\alpha}_n$ are given by \eqref{rel1-20231212}; $\widetilde{\beta}_1, \ldots, \widetilde{\beta}_{n-1}$ are given by \eqref{rel2-20231212} and $\widetilde{\beta}_{n, \theta}$ is defined in \eqref{rel3-20231212}. For the second summand, we use \eqref{relation1-20230831-3:59pm} and take power $\frac{1}{2H_0}$, we obtain:
\begin{align*}
\Big(I^{(n)}_{t}(\pmb{t_n}) \Big)^{\frac{1}{2H_0}} \le C_{H,1}^{\frac{n}{2H_0}} \sum_{\pmb{\alpha_n} \in D_{n}^{(H)}}  t^{\frac{1+\alpha_n}{4H_0}}  \prod_{k=1}^{n} t_k^{\widetilde{\alpha}_k} (t_{k+1} - t_k)^{\widetilde{\beta}_k} 
\end{align*}
where $\widetilde{\alpha}_k$ and $\widetilde{\beta}_k$ are defined the same as in \eqref{relation2-20230901-4:00pm}. 

Putting these estimates for the two terms together, we obtain that for any $0 < t_1 < \ldots < t_n < t_{n+1} = t$, there exists a constant $C$ depending on $a$, $b$, $t_0$, $\theta$ and $T$ such that 
\begin{align*}
&\Big( \Psi^{H}_{t,z,n}(\pmb{t_n}) \Big)^{\frac{1}{2H_0}} \le |z|^{\frac{\theta}{H_0}} C \bigg(   2^{\frac{n}{2H_0}} 
\sum_{\pmb{\alpha_n} \in D_{n}^{(H)}}  t^{\frac{1+\alpha_n}{4H_0}} \prod_{k = 1}^{n} t^{\widetilde{\alpha}_k}_{k}  \prod_{k = 1}^{n-1} \big( t_{k+1} - t_{k}\big)^{\widetilde{\beta}_k}
\big( t - t_{n}\big)^{\widetilde{\beta}_{n, \theta}} \\
&\qquad \qquad \qquad \qquad \qquad \qquad \qquad \qquad +
\sum_{\pmb{\alpha_n} \in D_{n}^{(H)}}  t^{\frac{1+\alpha_n}{4H_0}}  \prod_{k=1}^{n} t_k^{\widetilde{\alpha}_k} (t_{k+1} - t_k)^{\widetilde{\beta}_k} 
\bigg).
\stepcounter{equation}\tag{\theequation}\label{rel1-20231116-2:06pm} 
\end{align*}
A similar estimate holds for arbitrary $(t_1, \ldots, t_n) \in [0, t]^n$ with $t_{\rho(1)} < \ldots < t_{\rho(n)}$ for same permutation $\rho \in S_n$. 

We now integrate $\big( \Psi^{H}_{t,z,n}(\pmb{t_n}) \big)^{\frac{1}{2H_0}}$ over $[0,t]^n$. For this, we use Lemma \ref{lem1-20230823-11:07am}. For the first integral, we use \eqref{rel1-20231210-3:20pm} and for the second integral we use \eqref{relation2-20230901-8:06pm}. Note that, the polynomial function of $t$ in the right side of \eqref{rel1-20231210-3:20pm}, respectively \eqref{relation2-20230901-8:06pm}, is $t^{|\widetilde{\alpha}| + |\widetilde{\beta}| + n} =  t^{n \frac{2H_0 + H - 1}{2H_0} - \frac{1+\alpha_n}{4H_0}}$ and the factor $t^{-\frac{1+\alpha_n}{4H_0}}$ cancels when we multiplying the integral by $t^{\frac{1+\alpha_n}{4H_0}}$, as required when integrating the right hand side of \eqref{rel1-20231116-2:06pm} with respect to $\pmb{t_n}$. Using the same estimates above, we see that
\[
\int_{[0,t]^n}\Big( \Psi^{H}_{t,z,n}(\pmb{t_n}) \Big)^{\frac{1}{2H_0}}  d\pmb{t_n} \leq  |z|^{\frac{\theta}{H_0}} \frac{ \big( C^{(9)} \big)^{n-1} }{[(n-1)!]^{\frac{a-1}{2H_0}}}.
\]
We conclude that
\[
\sum_{n\geq 1}(p-1)^{n/2}\left( \frac{1}{n!}C_n^H(t,z)\right)^{1/2}\leq C |z|^{\theta},
\]
where $C$ is a constant depending on $t_0$, $T$, $p$, $a$, $b$, $c_0$ and $\theta$.

\end{proof}

{\bf Proof of Theorem \ref{main-th2}}: {\em Step 1. (finite dimensional convergence)} In this step, we prove that:
\[
\big(u^{H_n}(t_1,x_1),\ldots,u^{H_n}(t_k,x_k)\big)
\stackrel{d}{\to}
\big(u^{H^*}(t_1,x_1),\ldots,u^{H^*}(t_k,x_k)\big),
\]
for any $(t_1,x_1),\ldots,
(t_k,x_k)\in [0,T] \times \bR$. It is enough to prove that for any $(t,x) \in [0,T] \times \bR$, $u^{H_n}(t,x) \to u^{H^*}(t,x)$ in $L^2(\Omega)$ as $n\to \infty$.

As in the proof of Theorem \ref{main-th1}, we approximate $u^{H}(t,x)$ by the partial sum:
\[
u_m^H(t,x)=1+\sum_{k=1}^m I_{n}^H(f_{t,x,k}).
\]
Then $u_m^{H^{\ast}}(t,x) \to u^{H^{\ast}}(t,x)$ in $L^2(\Omega)$ as $m\to \infty$. Moreover, by Lemma \ref{rough-conv-Ik}, for any $m\geq 1$ fixed,
$u_m^{H_n}(t,x) \to u_m^{H*}(t,x)$ in $L^2(\Omega)$ as $n\to \infty$. So, it remains to prove that 
$$\sup_{n\geq 1}\bE |u_m^{H_n}(t,x)-u^{H_n}(t,x)|^2 \to 0$$ 
as $m\to \infty$. We choose values $a$ and $b$ such that \eqref{def-ab} holds. Since $H_n \to H^*$, there exists $N \in \bN$ such that $a<H_n<b$ for all $n\geq N$. Therefore, it suffices to prove that:
\begin{equation}
\label{rough-unif-conv}
\sup_{H \in [a,b]}\bE |u_m^H(t,x)-u^H(t,x)|^2=\sup_{H \in [a,b]}\sum_{k\geq m+1}\bE|I_k^{H}(f_{t,x,k})|^2 \to 0 \quad \mbox{as $m\to \infty$}.
\end{equation}

To prove \eqref{rough-unif-conv}, we need to estimate $\bE|I_k^{H}(f_{t,x,k})|^2$. Those estimates have been done in \cite{BCY2022}, but we need to revisit them here since we are interested in obtaining a uniform bound in $H$. First, note that by Littlewood-Hardy inequality \eqref{LH-ineq},
\begin{equation}
\label{bound-Ik}
\bE|I_k^{H}(f_{t,x,k})|^2 =k! \, \|\widetilde{f}_{t,x,k} \|_{\cH_{H}^{\otimes k}}^2 \leq k! \, b_{0}^k \left(\int_{[0,t]^k}
A_k^{H}(\pmb{t_k})^{\frac{1}{2H_0}}d\pmb{t_k}\right)^{2H_0},
\end{equation}
where $A_k^{H}(\pmb{t_k})=c_H^k \int_{\bR^k}|\cF \widetilde{f}_{t,x,k}(\pmb{t_k},\bullet)(\pmb{\xi_k})|^2 \prod_{j=1}^{k}|\xi_j|^{1-2H}d\pmb{\xi_k}$.

Next, we would like to find an upper bound for $A_k^{H}(\pmb{t_k})$. For this, using the same arguments as in the proof of Theorem 2.2 of \cite{BL2023} (the study of $A_k^{H}(\pmb{t_k})$), with $\cF G_{t_{\rho(j+1)}-t_{\rho(j)}}(\eta_j)$ replaced by $\cF G_{(t_{\rho(j+1)} - t_{\rho(j)})/t_{\rho(j+1)}t_{\rho(j)}}(\eta_j)$, we obtain:
\begin{align*}
&A_k^{H}(\pmb{t_k}) \leq \frac{c_H^k}{(k!)^2} \Big( \prod_{j = 1}^{k} t_{\rho(j)} \Big)^{2H - 2} 
\sum_{\pmb{\alpha}\in D_k^{(H)}}  \prod_{j=1}^{k} \bigg\{ \int_{\bR}  \exp \Big(   - \frac{t_{\rho(j+1)} - t_{\rho(j)}}{t_{\rho(j+1)}t_{\rho(j)}} | \eta_j |^2  \Big) |\eta_j|^{\alpha_j} d\eta_j \bigg\}.
\end{align*}
where $D_k^{(H)}$ be the set of multi-indices $\pmb{\alpha}=(\alpha_1,\ldots,\alpha_k)$ with $\alpha_j=(1-2H)a_j$ and $\pmb{a}=(a_1,\ldots,a_k) \in A_k$
Using Lemma \ref{lemma2-20220530-4:52}, we have:
\begin{align*}
A_k^{H}(\pmb{t_k}) &\le \frac{c_H^k}{(k!)^2}  \Big( \prod_{j = 1}^{k} t_{\rho(j)} \Big)^{2H - 2} 
\sum_{\pmb{\alpha}\in D_k^{(H)}}  \prod_{j = 1}^{k} \Gamma
\Big(\frac{1+ \alpha_j}{2} \Big) \Big( \frac{t_{\rho(j+1)} - t_{\rho(j)}}{t_{\rho(j+1)}t_{\rho(j)}} \Big)^{- \frac{1 + \alpha_k}{2}}\\
&\le \frac{C_{H,1}^k}{(k!)^2} \sum_{\pmb{\alpha}\in D_k^{(H)}}  t^{\frac{1 + \alpha_n}{2}} t_{\rho(1)}^{\frac{4H - 3 + \alpha_1}{2}} 
\Big( \prod_{j=2}^{k} t_{\rho(j)}^{\frac{4H - 2 + \alpha_{k-1} + \alpha_k }{2}} \Big) 
\prod_{j=1}^{k} (t_{\rho(j+1)} - t_{\rho(j)})^{- \frac{1 + \alpha_k}{2}}
%\stepcounter{equation}\tag{\theequation}\label{relation1-20230831-3:59pm}
\end{align*}
where for the second inequality we used the constant $C_{H,1}$ defined by \eqref{constantC_H_1}. 

Coming back to \eqref{bound-Ik}, we obtain that:
\begin{align*}
& \bE|I_k^H(f_{t,x,k})|^2 \leq b^2_{H_0} w_{+}^2(t,x) C_{H,1}^k (k!)^{2H_0 - 1} 
\bigg( \sum_{\pmb{\alpha}\in D_k^{(H)}}  t^{\frac{1+\alpha_n}{4H_0}} 
\int_{T_k(t)} \prod_{j=1}^{k} t_{j}^{\widetilde{\alpha}_j} (t_{j+1} - t_{j})^{\widetilde{\beta}_j} d\pmb{t_k} \bigg)^{2H_0}
\end{align*}
where $\widetilde{\alpha}_j$ and $\widetilde{\beta}_j$ for all $j = 1, \ldots, k$ are the same as in \eqref{relation2-20230901-4:00pm}. The last integral is calculated by Lemma \ref{lem1-20230823-11:07am}, the integral is
\begin{align*}
\int_{T_k(t)} \prod_{j=1}^{k} t_{\rho(j)}^{\widetilde{\alpha}_j} (t_{\rho(j+1)} - t_{\rho(j)})^{\widetilde{\beta}_j} d\pmb{t_k}
&= \frac{\Gamma(\widetilde{\alpha}_1 + 1) \prod_{j = 1}^{k}\Gamma(\widetilde{\beta}_j + 1)}{\Gamma \big(\sum_{j=1}^{k}\widetilde{\alpha}_j + \sum_{j=1}^{k}\widetilde{\beta}_j  + k + 1  \big)} \gamma_{k-1} 
t^{\sum_{j=1}^{k}\widetilde{\alpha}_j + \sum_{j=1}^{k}\widetilde{\beta}_j + k}
\end{align*}
Similarly as in \eqref{relation2-20230901-8:06pm}, those factors appears in above can be bounded uniformly in $H \in [a,b]$. There exists a constant $C > 0$ such that 
\begin{align*}
\sup_{H \in [a,b]} \sum_{k\ge m+1} \mathbb{E}\big| I^{H}_k(f_k(\cdot, t,x)) \big|^2  \le \sum_{k\ge m+1} C^k \frac{(t \vee 1)^{k (2H_0 + b - 1)}}{(k!)^{a} }.
\end{align*}
The last term clearly converges to $0$ as $m \to \infty$. 
This concludes the proof of \eqref{rough-unif-conv}.

\medskip

{\em Step 2. (tightness)} The fact that the sequence $(u^{H_n})_{n\geq 1}$ is tight in $C([t_0,T] \times \bR)$ follows by Proposition 2.3 of \cite{yor} using Theorem \ref{rough-unif-mom} above.

\qed

\medskip

\noindent{\em Acknowledgement}. The author is grateful to Raluca Balan for insightful discussions and for help in organizing and enhancing overall quality of this work.

%%%%%%%%%%%%%%%%%%%%%%%%%%%%%%%%%%%%%%%%%%%%%%%%

%%%%%%%%%%%%%%%%%%%%%%%%%%%%%%%%%%%%%%%%%%%%%%%%

%%%%%%%%%%%%%%%%%%%%%%%%%%%%%%%%%%%%%%%%%%%%%%%%

\end{document}